\documentclass[reqno,12pt]{amsart}
\usepackage[a4paper,height=20cm,top=26mm,bottom=26mm,left=26mm,right=26mm]{geometry}

\usepackage{amsmath,amssymb,amsthm,latexsym,booktabs,eepic}
\usepackage{amscd}
\usepackage
{graphicx}
\usepackage[all]{xy}
\usepackage{tikz}
\usetikzlibrary{calc,positioning}
\usepackage{subcaption}
\usepackage{cleveref}

\theoremstyle{definition}
\newtheorem{dfn}{Definition}[section]
\newtheorem{ex}[dfn]{Example}
\newtheorem{rem}[dfn]{Remark}
\theoremstyle{plain}
\newtheorem{prop}[dfn]{Proposition}
\newtheorem{thm}[dfn]{Theorem}
\newtheorem{lem}[dfn]{Lemma}
\newtheorem{cor}[dfn]{Corollary}

\crefname{dfn}{Definition}{Definitions}
\crefname{prop}{Proposition}{Propositions}
\crefname{thm}{Theorem}{Theorems}
\crefname{lem}{Lemma}{Lemmas}
\crefname{cor}{Corollary}{Corollaries}
\crefname{ex}{Example}{Examples}
\crefname{rem}{Remark}{Remarks}
\crefname{conj}{Conjecture}{Conjectures}
\crefname{fact}{Fact}{Facts}
\crefname{lemdef}{Lemma/Definition}{Lemma/Definitions}
\crefname{prob}{Problem}{Problems}
\crefname{figure}{Figure}{Figures}
\crefname{section}{Section}{Sections}
\crefname{appendix}{Appendix}{Appendix}

\Crefname{dfn}{Definition}{Definitions}
\Crefname{prop}{Proposition}{Propositions}
\Crefname{thm}{Theorem}{Theorems}
\Crefname{lem}{Lemma}{Lemmas}
\Crefname{cor}{Corollary}{Corollaries}
\Crefname{ex}{Example}{Examples}
\Crefname{rem}{Remark}{Remarks}
\Crefname{conj}{Conjecture}{Conjectures}
\Crefname{fact}{Fact}{Facts}
\Crefname{lemdef}{Lemma/Definition}{Lemma/Definitions}
\Crefname{prob}{Problem}{Problems}
\Crefname{figure}{Figure}{Figures}
\Crefname{section}{Section}{Sections}

\newcommand\Teich{{Teichm\"uller }}
\newcommand\surf{{F_{g, \vec{\delta}}^p}}
\newcommand\TF{{T(F)}}
\newcommand\wTF{{\widetilde{T}(F)}}
\newcommand\hTF{{\widehat{T}(F)}}
\newcommand\LQ{{\mathcal{L}(F;\mathbb{Q})}}
\newcommand\LR{{\mathcal{L}(F;\mathbb{R})}}
\newcommand\wLQ{{\widetilde{\mathcal{L}}(F;\mathbb{Q})}}
\newcommand\wLR{{\widetilde{\mathcal{L}}(F;\mathbb{R})}}
\newcommand\hLQ{{\widehat{\mathcal{L}}(F;\mathbb{Q})}}
\newcommand\hLR{{\widehat{\mathcal{L}}(F;\mathbb{R})}}

\newcommand\ML{{\mathcal{ML}_{0}^{+}(F)}}
\newcommand\PML{{\mathcal{PML}_{0}(F)}}
\newcommand\wML{{\mathcal{\widetilde{ML}}(F)}}

\newcommand\PSL{{PSL_{2}(\mathbb{R})}}
\newcommand\pos{{ \mathbb{R}_{>0} }}
\newcommand\trop{{ \mathbb{R}^t }}
\newcommand\Trop{{\mathrm{Trop} }}

\newcommand\MC{{MC(F)}}
\newcommand\Arc{{\mathrm{Arc}^{\bowtie}(F)}}
\newcommand\A{{\mathcal{A} }}
\newcommand\X{{\mathcal{X} }}
\newcommand\U{{\mathcal{U} }}
\newcommand\Z{{\mathcal{Z} }}
\newcommand\C{{\mathcal{C} }}
\newcommand\G{{\mathcal{G} }}
\newcommand\F{{\mathcal{F} }}
\newcommand\bS{{\mathbb{S} }}
\newcommand\Pos{{\mathrm{Pos} }}

\newcommand\bi{{\mathbf{i}}}

\makeatletter
    
    \@addtoreset{equation}{section}
  \makeatother

\begin{document}

\title[On a Nielsen-Thurston theory for cluster modular groups]
{On a Nielsen-Thurston classification theory for cluster modular groups}
\author{Tsukasa Ishibashi}

\address{Graduate School of Mathematical Sciences, the University of Tokyo, 3-8-1 Komaba, Meguro, Tokyo 153-8914, Japan}

\email{ishiba@ms.u-tokyo.ac.jp}
\maketitle

\begin{abstract}
We classify elements of a cluster modular group into three types. We characterize them in terms of fixed point property of the action on the tropical compactifications associated with the corresponding cluster ensemble. The characterization gives an analogue of the Nielsen-Thurston classification theory on the mapping class group of a surface.
\end{abstract}

\setcounter{section}{-1}
\tableofcontents

\setcounter{section}{-1}
\section{Introduction}
A \emph{cluster modular group}, defined in ~\cite{FG09}, is a group associated with a combinatorial data called a \emph{seed}. An element of the cluster modular group is a finite composition of permutations of vertices and \emph{mutations}, which preserves the \emph{exchange matrix} and induces non-trivial ($\A$- and $\X$-)\emph{cluster transformations}. The cluster modular group acts on the \emph{cluster algebra} as automorphisms (only using the $\A$-cluster transformations). A closely related notion of an automorphism group of the cluster algebra, which is called the \emph{cluster automorphism group}, is introduced in ~\cite{ASS} and further investigated by several authors ~\cite{Blanc-Dolgachev,CZ16,CZ16coeff,Lawson16}. Relations between the cluster modular group and the cluster automorphism group are investigated in ~\cite{Fraser}.

It is known that, for each marked hyperbolic surface $F$, the cluster modular group associated with the seed associated with an ideal triangulation of $F$ includes the \emph{mapping class group} of $F$ as a subgroup of finite index ~\cite{BS15}. Therefore it seems natural to ask whether a property known for mapping class groups holds for general cluster modular groups. In this paper we attempt to provide an analogue of the \emph{Nielsen-Thurston theory} ~\cite{Thurston,FLP} on mapping class groups, which classifies mapping classes into three types in terms of fixed point property of the action on the \emph{Thurston compactification} of the \Teich space. Not only is this an attempt at generalization, but also it is expected to help deepen understanding of mapping classes as cluster transformations. A problem, which is equivalent to classifying mapping classes in terms of the cluster transformations, was originally raised in ~\cite{PP93}. 

The \emph{cluster ensemble} associated with a seed, defined in ~\cite{FG09}, plays a similar role as the \Teich space when we study cluster modular groups. It can be thought of two spaces on which the cluster modular group acts. Technically, it consists of two functors $\psi_\A$, $\psi_\X: \G \to \Pos(\mathbb{R})$, called $\A$- and $\X$-spaces respectively, which are related by a natural transformation $p: \psi_\A \to \psi_\X$. Here the objects of the target category are split algebraic tori over $\mathbb{R}$, and the values of these functors patch together to form a pair of contractible manifolds $\A(\pos)$ and $\X(\pos)$, on which the cluster modular group acts analytically. These manifolds are naturally compactified to a pair of topological closed disks $\overline{\A}=\A(\pos) \sqcup P\A(\trop)$ and $\overline{\X}=\X(\pos) \sqcup P\X(\trop)$, called the \emph{tropical compactifications} ~\cite{FG16,Le16}, on which the actions of the cluster modular group extend continuously. These are algebraic generalizations of the \emph{Thurston compactifications} of \Teich spaces. In the case of the seed associated with a triangulated surface, $\U(\pos)=p(\A(\pos))$ is identified with the \Teich space, $\A(\pos)$ and $\X(\pos)$ are the \emph{decorated \Teich space} and the \emph{enhanced \Teich space} introduced by Penner ~\cite{Penner87} and Fock-Goncharov ~\cite{FG07}, respectively. The tropical compactification $\overline{\U}$ is identified with the Thurston compactification of the \Teich space ~\cite{FG16}. For an investigation of the action of the cluster modular group on $\U(\mathbb{Z}^t)$, see ~\cite{Mandel14}.

For each seed, a simplicial complex called the \emph{cluster complex}, defined in ~\cite{FZ03} and ~\cite{FG09}, admits a simplicial action of the cluster modular group. In the case of the seed associated with an ideal triangulation of a surface $F$, the cluster complex is a finite covering of the \emph{arc complex} of $F$.
In terms of the action on the cluster complex, we define three types of elements of the cluster modular group, called \emph{Nielsen-Thurston types}. They constitute an analogue of the classification of mapping classes.

\begin{dfn}[Nielsen-Thurston types: \cref{dfn; NT types}]
Let $\bi$ be a seed, $\C=\C_{|\bi|}$ the corresponding cluster complex and $\Gamma=\Gamma_{|\bi|}$ the corresponding cluster modular group. An element $\phi \in \Gamma$ is said to be
\begin{enumerate}
\item periodic if $\phi$ has finite order,
\item cluster-reducible if $\phi$ has a fixed point in the geometric realization $|\C|$ of the cluster complex, and
\item cluster-pseudo-Anosov (cluster-pA) if no power of $\phi$ is cluster-reducible.
\end{enumerate}
\end{dfn}
These types give a classification of elements of the cluster modular group in the sense that the cyclic group generated by any element intersects with at least one of these types. 
We have the following analogue of the classical Nielsen-Thurston theory for general cluster modular groups, which is the main theorem of this paper.

\begin{thm}[\cref{thm; main thm}]\label{thm: cluster modular}
Let $\bi$ be a seed of \Teich type (see \cref{dfn; Teich type}) and $\phi \in \Gamma_{|\bi|}$ an element. Then the followings hold.
\begin{enumerate}
\item The element $\phi \in \Gamma$ is periodic if and only if it has fixed points in $\A(\pos)$ and $\X(\pos)$.
\item The element $\phi \in \Gamma$ is cluster-reducible if and only if there exists a point $L \in \X(\trop)_+$ such that $\phi[L]=[L]$.
\item If the element $\phi \in \Gamma$ is cluster-pA, there exists a point $L \in \X(\trop) \backslash \X(\trop)_+$ such that $\phi[L]=[L]$.
\end{enumerate}
\end{thm}

We will show that the seeds of \Teich type include seeds of finite type, the seeds associated with triangulated surfaces, and the rank $2$ seeds of finite mutation type.

In the theorem above, we neither characterize cluster-pA elements in terms of fixed point property, nor describe the asymptotic behavior of the orbits as we do in the original Nielsen-Thurston classification (see \cref{classical NT}). However we can show the following asymptotic behavior of orbits similar to that of pA classes in the mapping class groups, for certain classes of cluster-pA elements.

\begin{thm}[cluster reductions and cluster Dehn twists: \cref{thm; cluster Dehn twists}]\label{thm; generalized Dehn}\ {}
\begin{enumerate}
\item Let $\bi$ be a seed, $\phi \in \Gamma_{|\bi|}$ be a cluster-reducible element. Then some power $\phi^l$ induces a new element in the cluster modular group associated with a seed which has smaller mutable rank $n$. We call this process the \emph{cluster reduction}. 
\item After a finite number of cluster reductions, the element $\phi^l$ induces a cluster-pA element.
\item Let $\bi$ be a skew-symmetric connected seed which has mutable rank $n \geq 3$, $\phi \in \Gamma_{|\bi|}$ an element of infinite order.
If some power of the element $\phi$ is cluster-reducible to rank $2$, then there exists a point $[G] \in P\A(\trop)$ such that we have
\[
\lim_{n \to \infty}\phi^{\pm n}(g)=[G] \text{  in $\overline{\A}$}
\]
for all $g\in \A(\pos)$.
\end{enumerate}
\end{thm}

We call a mapping class which satisfies the assumption of \cref{thm; generalized Dehn}(3) \emph{cluster Dehn twist}. Dehn twists in the mapping class groups are cluster Dehn twists. The above theorem says that cluster Dehn twists have the same asymptotic behavior of orbits on $\overline{\A}$ as Dehn twists. We expect that cluster Dehn twists together with seed isomorphisms generate cluster modular groups, as Dehn twists do in the case of mapping class groups. The generation of cluster modular groups by cluster Dehn twists and seed isomorphisms will be discussed elsewhere.

This paper is organized as follows. In \cref{section: definition}, we recall some basic definitions from ~\cite{FG09}. Here we adopt slightly different treatment of the frozen vertices and definition of the cluster complex from those of ~\cite{FZ03,FG09}. In \cref{section: NT types}, we define the Nielsen-Thurston types for elements of cluster modular groups and study the fixed point property of the actions on the tropical compactifications. Our basic examples are the seeds associated with triangulated surfaces, studied in \cref{section: Teich}. Most of the contents of this section seem to be well-known to specialists, but they are scattered in literature. Therefore we tried to gather results and give a precise description of these seeds. 
Other examples are studied in \cref{examples}.

\bigskip 

\noindent \textbf{Acknowledgement.}
I would like to express my gratitude to my advisor, Nariya Kawazumi, for helpful guidance and careful instruction. Also I would like to thank Toshiyuki Akita, Vladimir Fock, Rinat Kashaev, and Ken'ichi Ohshika for valuable advice and discussion. This work is partially supported by the program for Leading Graduate School, MEXT, Japan.

\section{Definition of the cluster modular groups}\label{section: definition}

\subsection{The cluster modular groups and the cluster ensembles}

We collect here the basic definitions on cluster ensembles and cluster modular groups. This section is based on Fock-Goncharov's seminal paper ~\cite{FG09}, while the treatment of frozen variables here is slightly different from them. In particular, the dimensions of the $\A$- and $\X$-spaces equal to the rank and the mutable rank of the seed, respectively. See \cref{def: ensembles}.

\begin{dfn}[seeds]
A \emph{seed} consists of the following data ${\bi} =(I, I_0, \epsilon, d)$;
\begin{enumerate}
\item $I$ is a finite set and $I_0$ is a subset of $I$ called the \emph{frozen subset}. An element of $I-I_0$ is called a \emph{mutable vertex}.
\item $\epsilon=(\epsilon_{ij})$ is a $\mathbb{Q}$-valued function on $I \times I$ such that $\epsilon_{ij} \in \mathbb{Z}$ for $(i, j) \notin I_0 \times I_0$, which is called the \emph{exchange matrix}.
\item $d = (d_i) \in \mathbb{Z}_{>0}^{I}$ such that $\mathrm{gcd}(d_i)=1$ and the matrix $\hat{\epsilon}_{ij}:= \epsilon_{ij}d_j$ is skew-symmetric.
\end{enumerate}
The seed $\bi$ is said to be \emph{skew-symmetric} if $d_i=1$ for all $i \in I$. In this case the exchange matrix $\epsilon$ is a skew-symmetric matrix. We simply write ${\bi} =(I, I_0, \epsilon)$ if $\bi$ is skew-symmetric. We call the numbers $N:=|I|$, $n:=|I-I_0|$ the \emph{rank} and the \emph{mutable rank} of the seed $\bi$, respectively. 
\end{dfn}
\begin{rem}
Note that unlike Fomin-Zelevinsky's definition of seeds (e.g. ~\cite{FZ03}), our definition does not include the notion of \emph{cluster variables}. A corresponding notion, which we call the \emph{cluster coordinate}, is given in \cref{seed tori} below. 
\end{rem}
Skew-symmetric seeds are in one-to-one correspondence with quivers without loops  and 2-cycles. Here a loop is an arrow whose endpoints are the same vertex, and a 2-cycle is a pair of arrows sharing both endpoints and having different orientations. Given a skew-symmetric seed ${\bi} =(I, I_0, \epsilon)$, the corresponding quiver is given by setting the set of vertices $I$, and drawing $|\epsilon_{ij}|$ arrows from the vertex $i$ to the vertex $j$ (resp. $j$ to $i$) if $\epsilon_{ij}>0$ (resp.  $\epsilon_{ij}<0$).

\begin{dfn}[seed mutations]
For a seed ${\bi} =(I, I_0, \epsilon, d)$ and a vertex $k \in I - I_0$, we define a new seed  ${\bi'} =(I', I_0', \epsilon', d')$ as follows:
\begin{itemize}
\item $I':=I, I_0':= I_0, d':=d$,
\item $\epsilon'_{ij}:= 
\begin{cases}
-\epsilon_{ij} & \text{if $k \in \{ i, j\}$}, \vspace{2mm} \\
\epsilon_{ij} + \displaystyle \frac{|\epsilon_{ik}|\epsilon_{kj}+  \epsilon_{ik}|\epsilon_{kj}|}{2} & \text{ otherwise}.
\end{cases}$
\end{itemize}
We write $\bi' = \mu_k(\bi)$ and refer to this transformation of seeds as the \emph{mutation directed to the vertex $k$}.
\end{dfn}

Next we associate \emph{cluster transformation} with each seed mutation. For a field $k$, let $k^*$ denote the multiplicative group. Our main interest is the case $k=\mathbb{R}$. A direct product $(k^*)^n$ is called a \emph{split algebraic torus} over $k$.
\begin{dfn}[seed tori]\label{seed tori}
Let $\bi=(I,I_0,\epsilon,d)$ be a seed and $\Lambda:= \mathbb{Z}[I ]$, $\Lambda':=\mathbb{Z}[I-I_0]$ be the lattices generated by $I$ and $I-I_0$, respectively.
\begin{enumerate}
\item $\X_{\bi}(k):= \mathrm{Hom}_\mathbb{Z}( \Lambda', k^*)$ is called the \emph{seed $\X$-torus} associated with $\bi$. For $i \in I - I_0$, the character $X_i: \X_{\bi} \to k^*$ defined by $\phi \mapsto \phi(e_i)$ is called the \emph{cluster $\X$-coordinate}, where $(e_i)$ denotes the natural basis of $\Lambda'$.
\item Let $f_i:=d_i^{-1} e_i^* \in \Lambda^*\otimes_{\mathbb{Z}}\mathbb{Q}$ and $\Lambda^\circ:=\oplus_{i \in I}\mathbb{Z}f_i \subset  \Lambda^*\otimes_{\mathbb{Z}}\mathbb{Q}$ another lattice, where $\Lambda^*$ denotes the dual lattice of $\Lambda$ and $(e_i^*)$ denotes the dual basis of $(e_i)$.
Then $\A_{\bi}(k):= \mathrm{Hom}_\mathbb{Z}( \Lambda^\circ, k^*)$ is called the \emph{seed $\A$-torus} associated with $\bi$. For $i \in I $, the character $A_i: \A_{\bi} \to k^*$ defined by $\psi \mapsto \psi(f_i)$ is called the \emph{cluster $\A$-coordinate}. 
The coordinates $A_i$ ($i \in I_0$) are called \emph{frozen variables}.
\end{enumerate}
\end{dfn}
Note that $\X_\bi(k) = (k^*)^n$ and $\A_\bi(k) = (k^*)^N$ as split algebraic tori. These two tori are related as follows. Let $p^*: \Lambda' \to \Lambda^\circ$ be the linear map defined by \[
p^*(v) = \sum_{\substack{i \in I-I_0 \\ k \in I}}  v_i \epsilon_{ik} f_k
\]
for $v=\sum_{i \in I-I_0} v_i e_i \in \Lambda'$. By taking $\mathrm{Hom}_\mathbb{Z}( -, k^*)$, it induces a monomial map $p_{\bi}: 
\A_{\bi} \to \X_{\bi}$, which is represented in cluster coordinates as $p_{\bi}^* X_i = \prod_{k \in I} A_k^{\epsilon_{ik}}$.

\begin{rem}
Note that we assign cluster $\X$-coordinates only on mutable vertices, which is a different convention from that of ~\cite{FG09}. It seems to be natural to adopt our convention from the point of view of the \Teich theory (see \cref{section: Teich}).
\end{rem}

\begin{dfn}[cluster transformations]\label{cluster transf}
For a mutation $\mu_k:\bi \to \bi'$, we define transformations on seed tori called the \emph{cluster transformations} as follows;

\begin{enumerate}
\item $\mu_k^{x}: \X_{\bi} \to \X_{\bi'}$, \\
$(\mu_k^{x})^* X_i':= 
\begin{cases}
X_k^{-1} & \text{if $i=k$}, \\
X_i(1+ X_k^{\mathrm{sgn}\epsilon_{ki}})^{\epsilon_{ki}} & \text{otherwise},
\end{cases}$
\item $\mu_k^{a}: \A_{\bi} \to \A_{\bi'}$, \\
$(\mu_k^{a})^* A_i':= 
\begin{cases}
A_i^{-1}(\prod_{\epsilon_{kj}>0} A_j^{\epsilon_{kj}} +
 \prod_{\epsilon_{kj}<0} A_j^{-\epsilon_{kj}}) & \text{if $i=k$}, \\
A_i & \text{otherwise}.
\end{cases}$
\end{enumerate}
\end{dfn}
Note that the frozen $\A$-variables are not transformed by mutations, while they have an influence on the transformations of the mutable $\A$-variables.

\begin{dfn}[the cluster modular group]
Let $\bi=(I,I_0,\epsilon,d)$ be a seed. Recall that a \emph{groupoid} is a small category whose morphisms are all invertible.
\begin{enumerate}
\item A \emph{seed isomorphism} is a permutation $\sigma$ of $I$ such that $\sigma(i)=i$ for all $i \in I_0$ and $\epsilon_{\sigma(i) \sigma(j)}=\epsilon_{ij}$ for all $i,j \in I$.
A \emph{seed cluster transformation} is a finite composition of mutations and seed isomorphisms. A seed cluster transformation is said to be \emph{trivial} if the induced cluster $\A$- and $\X$- transformations are both identity. Two seeds are called \emph{equivalent} if they are connected by a seed cluster transformation. Let $|\bi|$ denote the equivalence class containing the seed $\bi$.
\item Let $\G_{|\bi|}$ be the groupoid whose objects are seeds in $|\bi|$, and morphisms are seed cluster transformations, modulo trivial ones.
The automorphism group $\Gamma= \Gamma_{|\bi|}:= \mathrm{Aut}_{\G_{|\bi|}}(\bi)$ is called the \emph{cluster modular group} associated with the seed $\bi$. We call elements of the cluster modular group \emph{mapping classes} in analogy with the case in which the seed is coming from an ideal triangulation of a surface (see \cref{section: Teich}).
\end{enumerate}
\end{dfn}

\begin{ex}\label{example: cluster modular group}
We give some examples of cluster modular groups.
\begin{enumerate}
\item (Type $A_2$). Let ${\bi}:= (\{0,1\}, \emptyset, \epsilon)$ be the skew-symmetric seed defined by $\epsilon:= \begin{pmatrix}
0 & 1 \\ -1 & 0
 \end{pmatrix}$, which is called \emph{type $A_2$}. 
Let $\phi:=(0\ 1)\circ \mu_0 \in \Gamma_{A_2}$. It is the generator of the cluster modular group. The associated cluster transformations are described as follows:
\begin{align*}
\phi^*(A_0, A_1)&=\left(A_1, \frac{1+A_1}{A_0}\right), \\
\phi^*(X_0, X_1)&=(X_1(1+X_0), X_0^{-1}).
\end{align*}
Then one can check that $\phi$ has order $5$ by a direct calculation. See ~\cite{FG09}Section 2.5 for instance. In particular we have $\Gamma_{A_2} \cong \mathbb{Z}\slash 5$. 

\item (Type $L_k$ for $k \geq 2$). For an integer $k \geq 2$, let ${\bi}_k:= (\{0,1\}, \emptyset, \epsilon_k)$ be the skew-symmetric seed defined by $\epsilon_k:= \begin{pmatrix}
0 & k \\ -k & 0
 \end{pmatrix}$. Let us refer to this seed as the \emph{type $L_k$}. The quiver associated with the seed ${\bi}_k$ is shown in \cref{fig; Lk}.
Let $\phi:=(0\ 1)\circ \mu_0 \in \Gamma_{L_k}$. It is the generator of the cluster modular group. In this case, the associated cluster transformations are described as follows:
\begin{align*}
\phi^*(A_0, A_1)&=\left(A_1, \frac{1+A_1^k}{A_0}\right), \\
\phi^*(X_0, X_1)&=(X_1(1+X_0)^k, X_0^{-1}).
\end{align*}
It turns out that in this case the element $\phi$ has infinite order ~\cite{FZ03}. See \cref{example: periodic}.

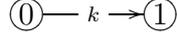
\begin{figure}
\[
\begin{xy}0;<1pt,0pt>:<0pt,-1pt>:: 
(50,0) *+[o][F]{0} ="0",
(100,0) *+[o][F]{1} ="1",
"0", {\ar|*+{\scriptstyle k}"1"},
\end{xy}
\]
\caption{quiver $L_k$}
\label{fig; Lk}
\end{figure}
\end{enumerate}
\end{ex}

Next we define the concept of a \emph{cluster ensemble}, which is defined to be a pair of functors related by a natural transformation. A cluster ensemble, in particular,  produces a pair of real-analytic manifolds, on which the cluster modular group acts analytically.

Let us recall some basic concepts from algebraic geometry.  For a split algebraic torus $H$, let $X_1,\dots,X_n$ be its coordinates. A rational function $f$ on $H$ is said to be  \emph{positive} if it can be represented as $f=f_1/f_2$, where $f_i=\sum_{\alpha \in \mathbb{N}^{n}} a_\alpha X^\alpha$ and $a_\alpha \in \mathbb{Z}_{\geq0}$. Here we write $X^\alpha:=X_1^{\alpha_1}\dots X_n^{\alpha_n}$ for a multi-index $\alpha \in \mathbb{N}^n$. Note that the set of positive rational functions on a split algebraic torus form a semifield under the usual operations. A rational map between two split algebraic tori $f: H_1 \to H_2$ is said to be \emph{positive} if the induced map $f^*$ preserves the semifields of positive rational functions.

\begin{dfn}[positive spaces]\ {}
\begin{enumerate}
\item Let $\Pos(k)$ be the category whose objects are split algebraic tori over $k$ and morphisms are positive rational maps. A functor $\psi: \G \to \Pos(k)$ from a groupoid $\G$ is called a \emph{positive space}.
\item A \emph{morphism} $\psi_1 \to \psi_2 $ between two positive spaces $ \psi_i: \G _i \to \Pos(k)$ (i=1,2) consists of the data $(\iota, p)$, where $\iota: \G_1 \to \G_2$ is a functor and $p: \psi_1 \Rightarrow \psi_2 \circ \iota$ is a natural transformation. A morphism of positive spaces $(\iota, p): \psi_1 \to \psi_2$ is said to be \emph{monomial} if the map between split algebraic tori $p_\alpha: \psi_1(\alpha) \to \psi_2(\iota(\alpha))$ preserves the set of monomials for each object $\alpha \in \G_1$.
\end{enumerate}
\end{dfn}

\begin{dfn}[cluster ensembles]\label{def: ensembles}\ {}
\begin{enumerate}
\item From \cref{cluster transf} we get a pair of positive spaces $\psi_{\X}, \psi_{\A}: \G_{|\bi|} \to \Pos(k)$, and we have a monomial morphism $p=p_{|\bi|}: \psi_{\A} \to \psi_{\X}$ (with $\iota=\mathrm{id}$), given by $p_{\bi}^* X_i = \prod_{k \in I} A_k^{\epsilon_{ik}}$ on each seed $\A$- and $\X$-tori. We call these data the \emph{cluster ensemble} associated with the seed $\bi$, and simply write as $p: \A \to \X$. The groupoid $\G=\G_{|\bi|}$ is called the \emph{coordinate groupoid} of the cluster ensemble. 
\item Let $\U=p(\A)$ be the positive space obtained by assigning the restriction $\psi_{\X}(\mu): p_{\bi}(\A_{\bi}) \to p_{\bi'}(\A_{\bi'})$ for each mutation $\mu: {\bi} \to {\bi'}$.
\end{enumerate}
\end{dfn}

\begin{dfn}[the positive real part]
For a cluster ensemble $p: \A \to \X$ and $\Z = \A,$ $\U$ or $\X$, define the \emph{positive real part} to be the real-analytic manifold obtained by gluing seed tori by corresponding cluster transformations, as follows:
\[
\left. \Z(\pos):=\bigsqcup_{\bi \in \G} \Z_{\bi}(\mathbb{R}_{>0}) \middle\slash (\mu_k^z) \right.,
\]
where $\Z_\bi(\mathbb{R}_{>0})$ denotes the subset of $\Z_\bi(\mathbb{R})$ defined by the condition that all cluster coordinates are positive.
Note that it is well-defined since positive rational maps preserves positive real parts. 
Similarly we define $\Z(\mathbb{Q}_{>0})$ and $\Z(\mathbb{Z}_{>0})$.
\end{dfn}
Note that we have a natural diffeomorphism $\Z_\bi(\pos) \to \Z(\pos)$ for each $\bi \in \G$. The inverse map $\psi_\bi^z : \Z(\pos) \to \Z_\bi(\pos)$ gives a chart of the manifold.
The cluster modular group acts on positive real parts $\Z(\mathbb{R}_{>0})$ as follows: 
\begin{equation}\label{action}
\xymatrix{
\Z(\pos) \ar[d]_{\phi} \ar[r]^{\psi_\bi^z}  & \Z_\bi(\pos) \ar[d]^{\mu_{i_1}^z \dots \mu_{i_k}^z \sigma^*} \\
\Z(\pos) \ar[r]^{\psi_\bi^z}  &  \Z_\bi(\pos) \\
}
\end{equation}
Here $\phi = \sigma \circ \mu_{i_k} \dots \mu_{i_1} \in \Gamma$ is a mapping class, $\sigma^*$ is the permutation of coordinates induced by the seed isomorphism $\sigma$. The fixed point property of this action is the main subject of the present paper.

\subsection{Cluster complexes}
We define a simplicial complex called the \emph{cluster complex}, on which the cluster modular group acts simplicially. In terms of the action on the cluster complex, we will define the Nielsen-Thurston types of mapping classes in \cref{section: NT types}. We propose here an intermediate definition between that of ~\cite{FZ03} and ~\cite{FG09}.

Let ${\bi}=(I, I_0, \epsilon, d)$ be a seed.
A decorated simplex is an ($n-1$)-dimensional simplex $S$ with a fixed bijection, called a \emph{decoration}, between the set of facets of S and $I-I_0$.
Let $\bS$ be the simplicial complex obtained by gluing (infinite number of) decorated $(n-1)$-dimensional simplices along mutable facets using the decoration. Note that the dual graph $\bS^{\vee}$ is a tree, and there is a natural covering from the set of vertices $V(\bS^{\vee})$ to the set of seeds. An edge of $\bS^{\vee}$ is projected to a mutation under this covering.
Assign mutable $\A$-variables to vertices of $\bS$ in such a manner that:
\begin{enumerate}
\item the reflection with respect to a mutable facet takes the $\A$-variables to the $\A$-variables which are obtained by the corresponding mutation.
\item the labels of variables coincide with the decoration assigned to the facet in the opposite side. 
\item the initial $\A$-coordinates are assigned to the initial simplex.
\end{enumerate}
Note that the assignment is well-defined since the dual graph $\bS^{\vee}$ is a tree.
Similarly we assign $\X$-variables to co-oriented facets of $\bS$ (see \cref{fig: assign}). 
Let $\Delta$ be the subgroup of $\mathrm{Aut}(\bS)$ which consists of elements that  preserve all cluster variables.
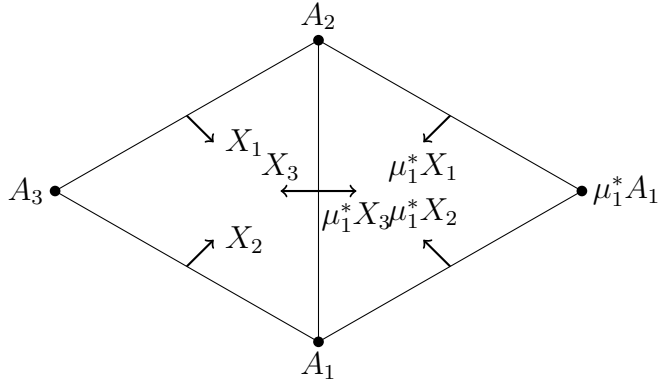
\begin{figure}
\begin{tikzpicture}
\fill (0,0) circle(2pt) coordinate(A);
\fill (0,4) circle(2pt) coordinate(B);
\fill (A) ++(150: 4) circle(2pt) coordinate(C);
\fill (A) ++(30: 4) circle(2pt) coordinate(D);
\draw (A) -- (B);
\draw (A) -- (C);
\draw (A) -- (D);
\draw (C) -- (B);
\draw (D) -- (B);
\path (A) node[below]{$A_1$};
\path (B) node[above]{$A_2$};
\path (C) node[left]{$A_3$};
\path (D) node[right]{$\mu_1^*A_1$};
\draw[->] (0,2) -- (-0.5,2) [thick] node[above]{$X_3$};
\draw[->] (0,2) -- (0.5,2) [thick] node[below]{$\mu_1^*X_3$};
\draw[->, thick] ($(C) !0.5! (B)$) -- ++(315: 0.5) node[right]{$X_1$};
\draw[->, thick] ($(D) !0.5! (B)$) -- ++(225: 0.5) node[below]{$\mu_1^*X_1$};
\draw[->, thick] ($(C) !0.5! (A)$) -- ++(45: 0.5) node[right]{$X_2$};
\draw[->, thick] ($(D) !0.5! (A)$) -- ++(135: 0.5) node[above]{$\mu_1^*X_2$};
\end{tikzpicture}
\caption{assignment of variables}
\label{fig: assign}
\end{figure}

\begin{dfn}[the cluster complex]\label{cluster complex}
The simplicial complex $\C=\C_{|\bi|}:= \bS \slash \Delta$ is called the \emph{cluster complex}.
A set of vertices $\{\alpha_1, \cdots, \alpha_n\} \subset V(\C)$ is called a \emph{cluster} if it spans a maximal simplex.
\end{dfn}
Let $\C^\vee$ denote the dual graph of the cluster complex. Note that the clusters, equivalently, the vertices of $\C^\vee$, are in one-to-one correspondence with seeds  together with tuples of mutable variables $((A_i), (X_i))$. For a vertex $v \in V(\C^\vee)$, let $[v]$ denote the underlying seed. Then we get coordinate systems of the positive real parts for each vertex $v \in V(\C^\vee)$, as follows:
\begin{align*}\xymatrix{
\psi_v^x: \X(\pos) \ar[r]^{\psi_{[v]}^x} & \X_\bi(\pos) \ar[r]^{(X_i)}  & \mathbb{R}^n_{>0} 
} \\
\xymatrix{
\psi_v^a: \A(\pos) \ar[r]^{\psi_{[v]}^a} & \A_\bi(\pos) \ar[r]^{(A_i)}  & \mathbb{R}^N_{>0} 
}
\end{align*}
The edges of $\C^\vee$ correspond to seed mutations, and the associated coordinate transformations are described by cluster transformations.

\begin{rem}
In ~\cite{FZ03}, the cluster complex is defined to be a simplicial complex whose ground set is the set of mutable $\A$-coordinates, while the definition in ~\cite{FG09} uses all (mutable/frozen) coordinates. In our definition, the frozen $\A$-variables have no corresponding vertices. The existence of the frozen variables does not change the structure of the cluster complex, see Theorem 4.8 of ~\cite{CKLP}.
\end{rem}

\begin{prop}[~\cite{FG09}Lemma 2.15]\label{action on cluster complex}
Let $D$ be the subgroup of $\mathrm{Aut}(\bS)$ which consists of elements which preserve the exchange matrix. Namely, an automorphism $\gamma$ belongs to $D$ if it satisfies $\epsilon^{[\gamma(v)]}_{\gamma(i), \gamma(j)}=\epsilon^{[v]}_{ij}$ for all $v \in V(\C^\vee)$ and $i,j \in [v]$. Then
\begin{enumerate}
\item $\Delta$ is a normal subgroup of $D$, and 
\item the quotient group $D \slash \Delta$ is naturally isomorphic to the cluster modular group $\Gamma$.
\end{enumerate}
In particular, the cluster modular group acts on the cluster complex simplicially.
\end{prop}

\begin{ex}\label{example: cluster complex}
The cluster complexes associated with seeds defined in \cref{example: cluster modular group} are as follows:
\begin{enumerate} 
\item (Type $A_2$). Let ${\bi}$ be the seed of type $A_2$. The cluster complex is a pentagon. The generator $\phi=(0\ 1)\circ \mu_0 \in \Gamma_{A_2}$ acts on the pentagon by the cyclic shift. 

\item (Type $L_k$ for $k \geq 2$). Let $\bi$ be the seed of type $L_k$. The cluster complex is 1-dimensional, and the generator $\phi=(0\ 1)\circ \mu_0 \in \Gamma_{L_k}$ acts by the shift of length $1$. The fact that $\phi$ has infinite order implies that the cluster complex is the line of infinite length. See \cref{example: periodic}.
\end{enumerate}
\end{ex}

\subsection{Tropical compactifications of positive spaces}

Next we define tropical compactifications of positive spaces, which are described in ~\cite{FG16}~\cite{Le16}. 

\begin{dfn}[the tropical limit]
For a positive rational map $f(X_1, \cdots , X_N)$ over $\mathbb{R}$, we define the \emph{tropical limit} $\Trop (f)$ of $f$ by
\[
\Trop (f)(x_1, \cdots, x_N):= \lim_{\epsilon \to 0} \epsilon \log f(e^{x_1/\epsilon}, \cdots, e^{x_N/\epsilon}),
\]
which defines a piecewise-linear function on $\mathbb{R}^N$.
\end{dfn}

\begin{dfn}[the tropical space]
Let $\psi_{\Z}: \G \to \Pos(\mathbb{R})$ be a positive space. Then let $\Trop(\psi_{\Z}): \G \to \mathrm{PL}$ be the functor given by the tropical limits of positive rational maps given by $\psi_{\Z}$, where $\mathrm{PL}$ denotes the category whose objects are euclidean spaces and morphisms are piecewise-linear $(PL)$ maps. Let $\Z(\trop)$ be the PL manifold obtained by gluing coordinate euclidean spaces by PL maps given by $\Trop(\psi_{\Z})$, which is called the \emph{tropical space}. 
\end{dfn}
Note that since PL maps given by tropical limits are homogeneous, $\mathbb{R}_{>0}$  naturally acts on $\Z(\trop)$. The quotient $P\Z(\trop):=(\Z(\trop) \backslash \{0\})\slash \mathbb{R}_{>0}$ is PL homeomorphic to a sphere. Let us denote the image of $G \in \Z(\trop) \backslash \{0\}$ under the natural projection by $[G] \in P\Z(\trop)$. The cluster modular group acts on $\Z(\trop)$ and $P\Z(\trop)$ by PL homeomorphisms, similarly as \cref{action}.

\begin{dfn}[a divergent sequence]\label{divergent}
For a positive space $\psi_{\Z}: \G \to \Pos(\mathbb{R})$, we say that a sequence $(g_m)$ in $\Z(\pos)$ is \emph{divergent} if for each compact set $K \subset \Z(\pos)$ there is a number $M$ such that $g_m \not\in K$ for all $m \geq M$.
\end{dfn}

\begin{dfn}[the tropical compactification]
Let $\psi_\X: \G \to \Pos(\mathbb{R})$ be the $\X$-space associated to a seed. For a vertex $v \in V(\C^\vee)$, let $\bi=[v]=(I,I_0, \epsilon, d)$ be the underlying seed, and $\psi_v^x$ and $\mathrm{Trop}(\psi_v^x)$ the associated positive and tropical coordinates, respectively. Then we define a homeomorphism $\F_v: \X(\pos) \to \X(\trop)$ by the following commutative diagram:
\[
\xymatrix{
\X(\pos) \ar[d]_{\F_v} \ar[r]^{\psi_v^x}  & \mathbb{R}_{>0}^n \ar[d]^{\log} \\
\X(\trop) \ar[r]^{\mathrm{Trop}(\psi_v^x)}  & \mathbb{R}^n \\
}.
\]
Fixing a vertex $v \in V(\C^\vee)$, we define the \emph{tropical compactification} by $\overline{\X}:= \X(\pos) \sqcup P\X(\trop)$, and endow it with the topology of the spherical compactification. Namely, a divergent sequence $(g_n)$ in $\X(\pos)$ converges to $[G] \in P\X(\trop)$ in $\overline{\X}$ if and only if $[\F_v(g_n)]$ converges to $[G]$  in $P\X(\trop)$.
Similarly we can consider the tropical compactifications of $\A$ and $\U$-spaces, respectively.
\end{dfn}

\begin{thm}[Le, ~\cite{Le16} Section 7]\label{Le}
Let $p:\A \to \X$ be a cluster ensemble, and $\Z=\A$, $\U$ or $\X$.
If we have $[\F_v(g_m)] \to [G]$ in $P\Z(\trop)$ for some $v \in V(\C^\vee)$, then we have $[\F_{v'}(g_m)] \to [G]$ in $P\Z(\trop)$ for all $v' \in V(\C^\vee)$. In particular the definition of the tropical compactification is independent of the choice of the vertex $v \in V(\C^\vee)$.
\end{thm}

\begin{cor}
Let $p: \A \to \X$ be a cluster ensemble, and $\Z=\A$, $\U$ or $\X$. Then the action of the cluster modular group on the positive real part $\Z(\pos)$ continuously extends to the tropical compactification $\overline{\Z}$.
\end{cor}

\begin{proof}
We need to show that $\phi_*(g_m) \to \phi_*([G])$ in $\overline{\Z}$ for each mapping class $\phi \in \Gamma$ and a divergent sequence $(g_m)$ such that $g_m \to [G]$ in $\overline{\Z}$. Here the action in the left-hand side is given by a composition of a finite number of cluster transformations and a permutation, while the action in the right-hand side is given by its tropical limit. Then the assertion follows from \cref{Le}.
\end{proof}
Note that each tropical compactification is homeomorphic to a closed disk of an appropriate dimension.

\section{Nielsen-Thurston types on cluster modular groups}\label{section: NT types}
In this section we define three types of elements of cluster modular groups in analogy with the classical Nielsen-Thurston types (see \cref{comparison}). Recall that the cluster modular group acts on the cluster complex simplicially.

\begin{dfn}[Nielsen-Thurston type]\label{dfn; NT types}
Let $\bi$ be a seed, $\C=\C_{|\bi|}$ be the corresponding cluster complex and $\Gamma=\Gamma_{|\bi|}$ the corresponding cluster modular group. An element $\phi \in \Gamma$ is called
\begin{enumerate}
\item periodic if $\phi$ has finite order,
\item cluster-reducible if $\phi$ has a fixed point in the geometric realization $|\C|$ of the cluster complex, and
\item cluster-pseudo-Anosov (cluster-pA) if no power of $\phi$ is cluster-reducible.
\end{enumerate}
\end{dfn}
Recall that the cluster modular group acts on the tropical compactifications $\overline{\A}=\A(\pos) \sqcup P\A(\trop)$ and $\overline{\X}=\X(\pos) \sqcup P\X(\trop)$, which are closed disks of dimension $N$ and $n$, respectively. Hence Brouwer's fixed point theorem says that each mapping class has at least one fixed point on each of the tropical compactifications.
The following is the main theorem of the present paper, which is an analogue of the classical Nielsen-Thurston classification theory.

\begin{thm}\label{thm; main thm}
Let $\bi$ be a seed and $\phi \in \Gamma_{|\bi|}$ a mapping class. Then the followings hold.
\begin{enumerate}
\item If the mapping class $\phi \in \Gamma$ is periodic, then it has fixed points in $\A(\pos)$ and $\X(\pos)$.

\item If the mapping class $\phi \in \Gamma$ is cluster-reducible, then there exists a point $L \in \X(\trop)_+$ such that $\phi[L]=[L]$.

\end{enumerate}
If the seed $\bi$ is of \Teich type (see \cref{dfn; Teich type}), the followings also hold:
\begin{enumerate}
\item[(1)$'$] if $\phi$ has a fixed point in $\A(\pos)$ or $\X(\pos)$, then $\phi$ is periodic.
\item[(2)$'$] if there exists a point $L \in \X(\trop)_+$ such that $\phi[L]=[L]$, then $\phi$ is cluster-reducible.
\end{enumerate}
\end{thm}

We prove the theorem in the following several subsections. 
The asymptotic behavior of orbits of certain type of cluster-pA classes on the tropical compactification of the $\A$-space will be discussed in \cref{sub: cluster Dehn twists}.

\subsection{Periodic classes}

Let us start by studying the fixed point property of periodic classes. Let $\Z=\A$ or $\X$. 

\begin{prop}\label{periodic}
Let $\bi$ be a seed, and $\Gamma=\Gamma_{|\bi|}$ the associated cluster modular group. 
For any $\phi \in \Gamma$, consider the following conditions:

\begin{enumerate}
{\renewcommand{\labelenumi}{(\roman{enumi})}
\item $\phi$ fixes a cell $C \in \C$ of finite type,
\item $\phi$ is periodic, and
\item $\phi$ has fixed points in $\Z(\pos)$}.
\end{enumerate}
Then we have $\mathrm{(i)} \Rightarrow (\mathrm{ii}) \Rightarrow (\mathrm{iii})$. Here a cell $C$ in the cluster complex is \emph{of finite type} if the set of supercells of $C$ is a finite set. 
\end{prop}

\begin{rem}
The converse assertion $\mathrm{(iii)} \Rightarrow (\mathrm{ii})$ holds under the  condition (T1) on the seed. See \cref{condition(T1)}.
\end{rem}

\begin{proof}

$(\mathrm{i}) \Rightarrow (\mathrm{ii})$.
Suppose we have $\phi(C)=C$ for some cell $C \in \C$ of finite type. Then from the definition, the set of supercells of $C$ is a finite set, and $\phi$ preserves this set. Since this set contains a maximal dimensional cell and the cluster complex $\C$ is connected, the action of $\phi$ on $\C$ is determined by the action on this finite set. Hence $\phi$ has finite order.

$(\mathrm{ii}) \Rightarrow (\mathrm{iii})$.
The proof is purely topological. Assume that $\phi$ has finite order. By Brouwer's fixed point theorem, $\phi$ has a fixed point on the disk $\overline{\Z} \approx D^N$. We need to show that there exists a fixed point in the interior $\Z(\pos)$. Suppose $\phi$ has no fixed points in the interior. Then $\phi$ induces a homeomorphism $\tilde{\phi}$ on the sphere $S^N=D^N \slash \partial D^N$ obtained by collapsing the boundary to a point, and $\tilde\phi$ has no fixed points other than the point corresponding to the image of $\partial D^N$. Now we use the following theorem.

\begin{thm}[Brown ~\cite{Brown82}Theorem 5.1]\label{Brown}
Let $X$ be a paracompact space of finite cohomological dimension, $s$ a homeomorphism of $X$, which has finite order. If $H_{*}(\mathrm{Fix}(s^k); \mathbb{Z})$ is finitely generated for each $k$, then the Lefschetz number of $s$ equals the Euler characteristic of the fixed point set:
\begin{equation}
\mathrm{Lef}(s):=\sum_{i} \mathrm{Tr}(s: H_i(X) \to H_i(X)) =\chi(\mathrm{Fix}(s)). \nonumber
\end{equation}
\end{thm}
Applying Brown's theorem for $X=S^N$ and $s=\tilde\phi$ we get a contradiction, since the Lefschetz number of $\tilde\phi$ is an even number in this case, while the Euler characteristic of a point is $1$. Indeed, the homology is non-trivial only for $i=0$ or $N$, and the trace equals to $\pm 1$ on each of these homology groups. Hence $\phi$ has a fixed point in the interior $\Z(\pos)$.
\end{proof}

To get the converse implication $(\mathrm{iii}) \Rightarrow (\mathrm{ii})$, we need a condition on the seed, which can be thought of an algebraic formulation of the proper discontinuity of the action of the cluster modular group on positive spaces.

\begin{prop}[Growth property $(\mathrm{T}1)$]\label{condition(T1)}
Suppose that a seed $\bi$ satisfies the following condition.
\begin{enumerate}
\item[(T1)]For each vertex $v_0 \in V(\C^\vee)$, $g \in \Z(\pos)$ and a number $M>0$, there exists a number $B>0$ such that $\max_{\alpha \in v} |\log Z_\alpha (g)| \geq M$ for all vertices $v \in V(\C^\vee)$ such that $[v]=[v_0]$ and $d_{\C^\vee}(v, v_0) \geq B$.
\end{enumerate}
Then the conditions $(\mathrm{ii})$ and $(\mathrm{iii})$ in \cref{periodic} are equivalent. Here $d_{\C^\vee}$ denotes the graph metric on the 1-skeleton of $\C^\vee$.
\end{prop}
Roughly speaking, the condition (T1) says that the values of the cluster coordinates evaluated at a point $g$ diverge as we perform a sequence of mutations which increase the distance $d_{\C^\vee}$.

\begin{proof}
Let $\phi \in \Gamma_{|\bi|}$ be an element of infinite order. We need to show that $\phi$ has no fixed points in $\Z(\pos)$. It suffices to show that each orbit is divergent. Let $g \in \Z(\pos)$ and $K \subset \Z(\pos)$ a compact set. We claim that there exists a number $M$ such that $\phi^m(g) \notin K$ for all $m \geq M$. Take a number $L>0$ so that $L > \max_{i=1,\dots,N}\max_{g \in K}|\log Z_i(g)|$.

Note that since the 1-skeleton of $\C^\vee$ has valency $n$ at any vertex, the graph metric $d_{\C^\vee}$ is proper. Namely, the number of vertices $v$ such that $d_{\C^\vee}(v, v_0) \leq B$ is finite for any $B >0$. Hence for the number $B>0$ given by the assumption (T1), there exists a number $M$ such that $d_{\C^\vee}(\phi^{-m}(v_0), v_0) \geq B$ for all $m \geq M$, since $\phi$ has infinite order. Also note that $[\phi^{-m}(v_0)]=[v_0]$ by \cref{action on cluster complex}. Then we have \[
\max_{i=1, \dots, N}|\log Z_i(\phi^m(g)|= \max_{\alpha \in \phi^{-m}(v_0)} |\log Z_\alpha(g)| \geq L
\]
for all $m \geq M$, where $(Z_1, \dots, Z_N)$ is the coordinate system associated with the vertex $v_0$. Here we used the equivariance of the coordinates $Z_{\phi^{-1}(\alpha)}(g)=Z_\alpha(\phi(g))$. Hence we have $\phi^m(g) \not\in K$ for all $m\geq M$.
\end{proof}

\begin{prop}\label{growth}
Assume that the cluster modular group $\Gamma_{|\bi|}$ acts on $\Z_{|\bi|}(\pos)$ proper discontinuously. Then the condition $(T1)$ holds.
\end{prop}

\begin{proof}
Suppose that the condition (T1) does not hold. Then there exists a vertex $v_0 \in V(\C^\vee)$, a point $g \in \Z(\pos)$, a number $M>0$, and a sequence $(v_m) \subset V(\C^\vee)$ such that $[v_m]=[v_0]$, $d_{\C^\vee}(v_m, v_0) \geq m$ and $\max_{\alpha \in v_m} |\log Z_\alpha (g)| \leq M$. Take a mapping class $\psi_m \in \Gamma$ so that $\psi_m(v_m)=v_0$. It is possible since $[v_m]=[v_0]$. Then we have
\[
\max_{i=1, \dots, N}|\log Z_i(\psi_m(g)|= \max_{\alpha \in \psi_m^{-1}(v_m)} |\log Z_\alpha(g)| \leq M,
\]
which implies that there exists a compact set $K \subset \Z(\pos)$ such that $\psi_m(g) \in K$ for all $n$. Note that the mapping classes $(\psi_m)$ are distinct, since the vertices $(v_m)$ are distinct. In particular we have $\psi_m^{-1}(K)\cap K \neq \emptyset$ for all $m$, consequently the action is not properly discontinuous.
\end{proof}

We will verify the condition (T1) for a seed associated with a triangulated surface using \cref{growth} in \cref{Teich are Teich}, and for the simplest case $L_k$ ($k \geq 2$) of infinite type in \cref{examples}. 

\begin{ex}\label{example: periodic}\ {}
\begin{enumerate}
\item (Type $A_2$). Let ${\bi}$ be the seed of type $A_2$ and $\phi=(0\ 1)\circ \mu_0 \in \Gamma_{A_2}$ the generator. See \cref{example: cluster modular group}. Recall that the two actions on the positive spaces $\A(\pos)$ and $\X(\pos)$ are described as follows:
\begin{align*}
\phi^*(A_0, A_1)&=\left(A_1,\frac{1+A_1}{A_0}\right), \\
\phi^*(X_0, X_1)&=(X_1(1+X_0), X_0^{-1}).
\end{align*}
The fixed points are given by $(A_0,A_1)=((1+\sqrt{5})/2, (1+\sqrt{5})/2)$ and $(X_0,X_1)=((1+\sqrt{5})/2, (-1+\sqrt{5})/2)$, respectively.

\item (Type $L_k$ for $k \geq 2$). Let $\bi$ be the seed of type $L_k$ and $\phi=(0\ 1)\circ \mu_0 \in \Gamma_{L_k}$ the generator. See \cref{example: cluster modular group}. Recall that the two actions on the positive spaces $\A(\pos)$ and $\X(\pos)$ are described as follows:
\begin{align*}
\phi^*(A_0, A_1)&=\left(A_1,\frac{1+A_1^k}{A_0}\right), \\
\phi^*(X_0, X_1)&=(X_1(1+X_0)^k, X_0^{-1}).
\end{align*}
These equations have no positive solutions. Indeed, the $\X$-equation implies $X_0^2=(1+X_0)^k$, which has no positive solution since $\begin{pmatrix}k \\ 2 \end{pmatrix} \geq 1$ for $k \geq 2$.
Similarly for $\A$-variables. Hence we can conclude that $\phi$ has infinite order by \cref{periodic}.  In particular we have $\Gamma_{L_k} \cong \mathbb{Z}$.

\end{enumerate}
\end{ex}

\subsection{Cluster-reducible classes}
In this subsection, we study the fixed point property of a cluster-reducible class. Before proceeding, let us mention the basic idea behind the constructions. Consider the seed associated with an ideal triangulation of a marked hyperbolic surface $F$. Here we assume $F$ is a closed surface with exactly one puncture or a compact surface without punctures (with marked points on its boundary). Then the vertices of the cluster complex $\C$ are represented by \emph{ideal arcs} on $F$. See \cref{Arc}. In particular each point in the geometric realization $|\C|$ of the cluster complex is represented by the projective class of a linear combination of ideal arcs. On the other hand, the Fock-Goncharov boundary $P\X(\trop)$, which is identified with the space of \emph{measured laminations} on $F$, contains all such projective classes. Hence the cluster complex is embedded into the Fock-Goncharov boundary of the $\X$-space in this case. In \cref{subsub: redX} we show that this picture is valid for a general seed satisfying some conditions. See \cref{isomorphism}. 

\subsubsection{Fixed points in the tropical $\X$-space}\label{subsub: redX}

\begin{dfn}[the non-negative part]
Let $\bi$ be a seed.
For each vertex $v \in V(\C^\vee)$, let $K_{v}:=\{ L \in \X(\trop) | L \geq 0\text{ in } v\}$ be a cone in the tropical space, where $L \geq 0\text{ in } v$ means that $x_\alpha(L)\geq 0$ for all $\alpha  \in v$. Then the union $\X(\trop)_+:= \bigcup_{v \in V(\C^\vee)} K_v \subseteq \X(\trop)$ is called the \emph{non-negative part} of $\X(\trop)$.

\end{dfn}

Let us define a $\Gamma$-equivariant map $\Psi: \C \to P\X(\trop)_+$ as follows. The construction contains reformulations of some conjectures stated in ~\cite{FG09} Section 5, for later use.
For each maximal simplex $S$ of $\bS$, let $[S]$ denote the image of $S$ under the projection $\bS \to \C$, and let $v \in V(\C^\vee)$ be the dual vertex of $[S]$. By using the barycentric coordinate of the simplex $S$, we get an identification $S\cong P\mathbb{R}^{n}_{\geq 0}$. Then we have the following map:
\[\xymatrix{
\Psi_S: S \cong P\mathbb{R}^{n}_{\geq 0} \ar[r]^{\xi_v^{-1}} & PK_v \subseteq P\X(\trop)_+,
}
\]
where $\xi_v:=\Trop(\psi_v^x): \X(\trop) \to \mathbb{R}^n $ is the tropical coordinate associated with the vertex $v$, whose restriction gives a bijection $K_v \to \mathbb{R}_{\geq 0}^n$. Since the tropical $\X$-transformation associated to a mutation $\mu_k: v \to v'$ preserves the set $\{ x_k=0\}$ and the dual graph $\bS^{\vee}$ is a tree, these maps combine to give a map 
\[
\Psi:=\bigcup_{v \in V(\C^\vee)} \Psi_v: \bS \to P\X(\trop)_+,
\]
which is clearly surjective. 
Assume we have $S'=\gamma(S)$ for some $\gamma \in \Delta$. Then from the definition of $\Delta$, $\gamma$ preserves all the tropical $\X$-coordinates. Hence we have $\Psi_{v'}(\gamma x)=\Psi_v(x)$ for all $x \in S$, and the map descends to
\[
\Psi : \C=\bS \slash \Delta \to P\X(\trop)_+.
\]

\begin{lem}\label{Psi equivariance}
The surjective map $\Psi$ defined above is $\Gamma$-equivariant.
\end{lem}

\begin{proof}
It follows from the following commutative diagram for $\phi \in \Gamma$: 
\[
\xymatrix{
S \ar[d]_{\phi} \ar[r]^{\cong}  & P\mathbb{R}^n_{ \geq 0}  \ar[d]^{\phi^*} \ar[r]^{\xi_v^{-1}} & K_v \ar[d]^{\phi^x} \\
\phi(S) \ar[r]^{\cong} & P\mathbb{R}^n_{ \geq 0}  \ar[r]^{\xi_{\phi^{-1}(v)}^{-1}} & PK_v \\
}
\]
Here $v$ is the dual vertex of $[S]=[\phi(S)]$, $\phi^*$ is the permutation on vertices induced by $\phi$, and $\phi^x$ is the induced tropical $\X$-transformation on $\X(\trop)$.
\end{proof}

Next we introduce a sufficient condition for $\Psi$ being injective.
For a point $L \in \X(\trop)$, a cluster $C$ in $\C$ is called a \emph{non-negative cluster} for $L$ if $L \in K_v$, where $v \in V(\C^\vee)$ is the dual vertex of $C$. The  subset $Z(L):=\{ \alpha\in V(C) \mid \xi_v(L; \alpha)=0\} \subset V(C)$ is called the \emph{zero subcluster} of $L$. Here $\xi_v(-; \alpha)$ denotes the component of the chart $\xi_v$ corresponding to the vertex $\alpha$. Since the mutation directed to a vertex $k \in Z(L)$ preserves the signs of coordinates, the cluster $\mu_k(C)$ inherits the zero subcluster $Z(L)$. Two non-negative clusters $C$ and $C'$ are called \emph{$Z(L)$-equivalent} if they are connected by a finite sequence of mutations directed to the vertices in $Z(L)$.

\begin{lem}\label{isomorphism}
Assume that a seed $\bi$ satisfies the following condition:
\begin{enumerate}
\item[(T2)] For each $L \in \X(\trop)_+$, any two non-negative clusters for $L$ are $Z(L)$-equivalent.
\end{enumerate}
Then the map $\Psi: \C \to P\X(\trop)_+$ is a $\Gamma$-equivariant isomorphism.
\end{lem}
Compare the condition $(\mathrm{T2})$ with Conjecture 5.10 in ~\cite{FG09}.

\begin{proof}
We need to prove the injectivity of $\Psi$. Note that $\Psi$ is injective on each simplex. Also note that, by the construction of the map $\Psi$, a point $[L]=\Psi(x)$ ($x \in C$) satisfies $Z(L) \neq \emptyset$ if and only if $x$ is contained in the boundary of the simplex $C$.  

Assume that $C$, $C'$  are distinct clusters, $x \in C$, $x' \in C'$ and $\Psi(x)=\Psi(x')=: [L] \in P\X(\trop)_+$. If $x$ lies in the interior of the cluster $C$, then $Z(L)= \emptyset$. Then the condition (T2) implies that $C=C'$, which is a contradiction. Hence $Z(L) \neq \emptyset$. Then the condition (T2) implies that $C'$ is $Z(L)$-equivalent to $C$. On the other hand, the point $x$ (resp. $x'$) must be contained in the face of $C$ (resp. $C'$) spanned by the vertices in $Z(L)$. Hence $x, x' \in Z(L) \subset C \cap C'$. In particular $x$ and $x'$ are contained in the same simplex, hence we have $x=x'$. Therefore $\Psi$ is injective.
\end{proof}

\begin{ex}Seeds of finite type satisfy the equivalence property $(\mathrm{T2})$, see ~\cite{FG09} Theorem 5.8.
\end{ex}

\begin{prop}[fixed points in $\X$-space]\label{redX}
Let $\bi$ be a seed, and $\phi \in \Gamma_{|\bi|}$ a mapping class. Then the followings hold.
\begin{enumerate}
\item If $\phi$ is cluster-reducible, then there is a point $L \in \X(\trop)_+ \backslash \{0\}$ such that $\phi[L]=[L]$.
\item If $\bi$ satisfies the condition $(\mathrm{T2})$, then the converse of $(1)$ is also true.
\end{enumerate}
\end{prop}

\begin{proof}
The assertions follow from \cref{Psi equivariance} and \cref{isomorphism}, respectively.
\end{proof}

\begin{dfn}[seeds of definite type]
A seed $\bi$ is \emph{of definite type} if $\X_{|\bi|}(\trop)_+ = \X_{|\bi|}(\trop)$.
\end{dfn}

\begin{prop}\label{prop: definite}
Assume that a seed $\bi$ satisfies the equivalence property $(\mathrm{T2})$.
Then $\bi$ is of definite type if and only if it is of finite type.
\end{prop}

\begin{proof}
The fact that finite type seeds are definite is due to Fock-Goncharov ~\cite{FG09}. Let us prove the converse implication. Assume that $\X$ is of definite type. Then by \cref{isomorphism} we have a homeomorphism $\Psi: \C \to P\X(\trop)$,  and the latter is homeomorphic to a sphere. In particular $\C$ is a compact simplicial complex, hence it can possess finitely many cells.
\end{proof}

\begin{rem}
The conclusion part in \cref{prop: definite} is Conjecture 5.7 in ~\cite{FG09}.
\end{rem}

\begin{dfn}[seeds of \Teich type]\label{dfn; Teich type}
A seed $\bi$ is of \Teich type if it satisfies the the growth property $(\mathrm{T1})$ and equivalence property $(\mathrm{T2})$, defined in \cref{condition(T1)} and \cref{isomorphism}, respectively.
\end{dfn}

\begin{ex}\ {}
\begin{enumerate}
\item Seeds of finite type are of \Teich type. See ~\cite{FG09} Theorem 5.8.
\item Seeds associated with triangulated surfaces are of \Teich type. See \cref{Teich are Teich}.
\item The seed of type $L_k$ ($k \geq 1$) is of \Teich type. See \cref{examples}. 
\end{enumerate}
\end{ex}

\begin{cor}\label{cor; cluster-pA}
Let $\bi$ be a seed of \Teich type, and $\phi \in \Gamma$ a cluster-pA class. Then there exists a point $L \in \X(\trop) \backslash \X(\trop)_+$ such that $\phi[L]=[L]$.
\end{cor}

\begin{proof}
Since the tropical compactification $\overline{\X}$ is a closed disk, Brouwer's fixed point theorem says that there exists a point $x \in \overline{\X}$ such that $\phi(x)=x$. If $x \in \X(\pos)$, then by assumption $\phi$ has finite order, which is a contradiction. If $x \in P\X(\trop)_+$, then by \cref{redX}(2), $\phi$ is cluster-reducible, which is a contradiction. Hence $x \in P(\X(\trop) \backslash \X(\trop)_+)$.
\end{proof}

\subsubsection{Cluster reduction and fixed points in the tropical $\A$-space}\label{subsub: redA}
Here we define an operation, called the \emph{cluster reduction}, which produces a new seed from a given seed and a certain set of vertices of the cluster complex. 
At the end of Section \ref{subsub: redA} we study the fixed point property of a cluster-reducible class on the tropical $\A$-space.

Let $\{\alpha_1, \dots, \alpha_k\} \subset V(\C)$ be a subset of vertices, which is contained in a cluster.

\begin{lem}[the cluster reduction of a seed]
Take a cluster containing $\{\alpha_1, \dots, \alpha_k\}$. Let ${\bi}=(I, I_0, \epsilon, d)$ be the underlying seed and $i_j:=[\alpha_j] \in I$ the corresponding vertex for $j=1,\dots, n-2$ under the projection $[\ ]: \{\text{clusters}\} \to \{\text{seeds}\}$ (see \cref{cluster complex}). Then we define a new seed by ${\bi}':= (I, I_0 \sqcup \{i_1, \dots, i_k\}, \epsilon, d)$, namely, by "freezing" the vertices $\{i_1, \dots, i_k\}$. Then the corresponding cluster complex $\C':=\C_{|\bi'|}$ is naturally identified with the link of $\{\alpha_1, \dots, \alpha_k\}$ in $\C$. In particular the equivalence class $|\bi'|$ does not depend on the choice of the cluster containing $\{\alpha_1, \dots, \alpha_k\}$. 
\end{lem}

\begin{proof}
Let $C \subset \C$ be a cluster containing $\{\alpha_1, \dots, \alpha_k\}$, $\bi=(I,I_0,\epsilon,d)$ the corresponding seed.
For a mutation directed to a mutable vertex $k \in I-(I_0 \sqcup \{i_1, \dots, i_k\})$, the cluster $C'=\mu_k(C)$ also contains $\{\alpha_1, \dots, \alpha_k\}$. Conversely, any cluster $C'$ containing $\{\alpha_1, \dots, \alpha_k\}$ is obtained by such a sequence of mutations. Hence each cluster in the cluster complex $\C'$ has the form $C \backslash \{\alpha_1, \dots, \alpha_k\}$, for some cluster $C \subset \C$ containing $\{\alpha_1, \dots, \alpha_k\}$.
\end{proof}
We say that the corresponding object, such as the cluster ensemble $p_{|\bi'|}:\A_{|\bi'|} \to \X_{|\bi'|}$ or the cluster modular group $\Gamma_{|\bi'|}$, is obtained by the  \emph{cluster reduction} with respect to the invariant set $\{\alpha_1, \dots, \alpha_k\}$ from the original one.
Next we show that some power of a cluster-reducible class induces a new mapping class by the cluster reduction.
\begin{lem}
Let $\bi$ be a seed, $\phi \in \Gamma_{|\bi|}$ a mapping class. Then $\phi$ is cluster-reducible if and only if it has an invariant set of vertices $\{\alpha_1, \dots, \alpha_k\} \in V(\C)$ contained in a cluster.
\end{lem}

\begin{proof}
Suppose $\phi$ is cluster-reducible. Then $\phi$ has a fixed point $c \in |\C|$. Since the action is simplicial, $\phi$ fixes the cluster $C$ containing the point $c$. Hence $\phi$ permute the vertices of $C$, which give an invariant set contained in $C$. The converse is also true, since $\phi$ fixes the point given by the barycenter of the vertices $\{\alpha_1, \dots, \alpha_k\}$.
\end{proof}

\begin{dfn}[proper reducible classes]\label{def: proper reducible}
A mapping class $\phi \in \Gamma_{|\bi|}$ is called \emph{proper reducible} if it has a fixed point in $V(\C)$.
\end{dfn}

\begin{lem}
Let $\phi \in \Gamma_{|\bi|}$ be a mapping class.
\begin{enumerate}
\item If $\phi$ is proper reducible, then $\phi$ is cluster reducible.
\item If $\phi$ is cluster-reducible, then some power of $\phi$ is proper reducible.
\end{enumerate}
\end{lem}

\begin{proof}
Clear from the previous lemma.
\end{proof}

\begin{lem}[the cluster reduction of a proper reducible class]
Let $\phi \in \Gamma_{|\bi|}$ be a proper reducible class, $\{\alpha_1, \dots, \alpha_k\}$ a fixed point set of vertices contained in a cluster. Then $\phi$ induces a new mapping class $\phi' \in \Gamma_{|\bi'|}$ in the cluster modular group obtained by the cluster reduction with respect to $\{\alpha_1, \dots, \alpha_k\}$. 
\end{lem}

\begin{proof}
The identification of $\C'$ with the link of the invariant set $\{\alpha_1, \dots, \alpha_k\}$ in $\C$ induces an group isomorphism \[
\Gamma_{|\bi'|} \cong \{ \psi \in \Gamma_{|\bi|} | \psi(\C')=\C'\},
\]
and the right-hand side contains $\phi$. Let $\phi' \in \Gamma_{|\bi'|}$ be the corresponding element. Note that $\phi$ fixes all frozen vertices in $\bi'$, since it is proper reducible.
\end{proof}
We say that the mapping class $\phi'$ is obtained by the \emph{cluster reduction} with respect to the fixed point set $\{\alpha_1, \dots, \alpha_k\}$ from $\phi$. 

\begin{lem}
A proper reducible class of infinite order induces a cluster-pA class in the cluster modular group corresponding to the seed obtained by a finite number of the cluster reductions.
\end{lem}

\begin{proof}
Clear from the definition of the cluster-pA classes.
\end{proof}

\begin{ex}[Type $X_7$]
Let $\bi=(\{0,1,2,3,4,5,6\}, \emptyset,\epsilon)$ be the skew-symmetric seed defined by the quiver described in \cref{fig; X7}. We call this seed \emph{type $X_7$}, following ~\cite{Derksen-Owen}. See also ~\cite{FeST12}. The mapping class $\phi_1:=(1\ 2)\circ \mu_1 \in \Gamma_{X_7}$ is proper reducible and fixes the vertex $A_i \in V(\C)$ ($i=0,3,4,5,6$), which is the $i$-th coordinate in the initial cluster. The cluster reduction with respect to the invariant set $\{A_0,A_3,A_4,A_5,A_6\}$ produces a seed $\bi'=(\{0,1,2,3,4,5,6\}, \{0,3,4,5,6\},\epsilon)$ of type $L_2$, except for some non-trivial coefficients. The cluster complex $\C_{|\bi'|}$ is identified with the link of $\{A_0,A_3,A_4,A_5,A_6\}$, which is the line of infinite length. The cluster reduction $\phi'$ is cluster-pA, and acts on this line by the shift of length $1$. Compare with \cref{example: cluster complex}.

The mapping class $\psi_1:=(0\ 1\ 2)(3\ 4\ 5\ 6)\circ \mu_2\mu_1\mu_0 \in \Gamma_{X_7}$ is cluster-reducible, since it has an invariant set $\{A_3,A_4,A_5,A_6\}$ contained in the initial cluster. Note that the power $\psi_1^2$ is proper reducible, since it fixes the vertex $A_0$.
\end{ex}

\begin{figure}
\unitlength 0.5mm
\begin{center}
\[
\begin{xy} 0;<0.5pt,0pt>:<0pt,-0.5pt>:: 
(140,100) *+[o][F]{0} ="0",
(0,50) *+[o][F]{1} ="1",
(75,0) *+[o][F]{2} ="2",
(225,0) *+[o][F]{3} ="3",
(280,50) *+[o][F]{4} ="4",
(175,240) *+[o][F]{5} ="5",
(105,240) *+[o][F]{6} ="6",
"0", {\ar"1"},
"2", {\ar"0"},
"0", {\ar"3"},
"4", {\ar"0"},
"0", {\ar"5"},
"6", {\ar"0"},
"1", {\ar|*+{\scriptstyle 2}"2"},
"3", {\ar|*+{\scriptstyle 2}"4"},
"5", {\ar|*+{\scriptstyle 2}"6"},
\end{xy}
\]
\caption{quiver $X_7$}
\label{fig; X7}
\end{center}

\end{figure}
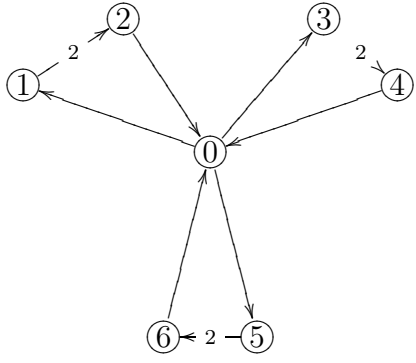

\begin{lem}\label{embedding}
Let $\bi$ be a seed, and $\bi'$ the seed obtained by a cluster reduction. Let $\psi_\A : \G_{|\bi|} \to \Pos(\mathbb{R})$ and $\psi'_\A : \G_{|\bi'|} \to \Pos(\mathbb{R})$ be the positive $\A$-spaces associated with the seeds $\bi$ and $\bi'$, respectively. Then there is a natural morphism of the positive spaces $(\iota, q): \psi'_\A \to \psi_\A$ which induces $\Gamma_{|\bi'|}$-equivariant homeomorphisms $\A'(\pos) \cong \A(\pos)$ and  $\A'(\trop) \cong \A(\trop)$.
\end{lem}

\begin{proof}
Note that the only difference between the two positive $\A$-spaces is the admissible directions of mutations. The functor $\iota: \G_{|\bi'|} \to \G_{|\bi|}$ between the coordinate groupoids is defined by $(I, I_0 \sqcup \{i_1, \dots, i_k\}, \epsilon, d) \mapsto (I, I_0, \epsilon, d)$ and sending the morphisms naturally. The identity map $\A_{\bi'}(k)=\A_\bi(k)$ for each $\A$-torus combine to give a natural transformation $q: \psi'_\A \Rightarrow \psi_\A \circ \iota$. The latter assertion is clear.
\end{proof}

\begin{rem}
We have no natural embedding of the $\X$-space in general, since $\X$-coordinates assigned to the vertices in $\{i_1, \dots, i_k\}$ may be changed by cluster $\X$-transformations directed to the vertices in $I-(I_0 \sqcup \{i_1, \dots, i_k\})$.
\end{rem}

\begin{dfn}
A tropical point $G \in \A(\trop)$ is said to be \emph{cluster-filling} if it satisfies $a_\alpha(G)\neq0$ for all $\alpha \in V(\C)$.
\end{dfn}
Note that the definition depends only on the projective class of $G$.

\begin{prop}[fixed points in $\A$-space]\label{redA}
Let $\bi$ be a seed satisfying the condition $(\mathrm{T1})$, and $\Gamma=\Gamma_{|\bi|}$ the corresponding cluster modular group. 
For a proper reducible class $\phi \in \Gamma$ of infinite order, there exists a non-cluster-filling point $G \in \A(\trop)$ such that $\phi[G]=[G]$.
\end{prop}

\begin{proof}
Let $\{\alpha_1,\dots, \alpha_k\}$ be a fixed point set of $\phi$ contained in a cluster, and $\phi' \in \Gamma_{|\bi'|}$ the corresponding cluster reduction. Since the tropical compactification $\overline{\A'}$ is a closed disk, $\phi'$ has a fixed point $x' \in \overline{\A'}$ by Brouwer's fixed point theorem. By \cref{condition(T1)}, $x$ must be a point on the boundary $P\A'(\trop)$. 
Then $\phi$ fixes the image $x$ of $x' \in \overline{\A}$ under the homeomorphism given by \cref{embedding}.  
\end{proof}

\subsection{Cluster-pA classes of special type: cluster Dehn twists}\label{sub: cluster Dehn twists}
Using the cluster reduction we define special type of cluster-pA mapping classes, called \emph{cluster Dehn twists}, and prove that they have an asymptotic behavior of orbits on the tropical compactification of the $\A$-space analogous to that of Dehn twists in the mapping class group.

\begin{dfn}[cluster Dehn twists]
Let $\bi$ be a skew-symmetric seed of mutable rank $n$.
A cluster-reducible class $\phi \in \Gamma_{|\bi|}$ is said to be \emph{cluster-reducible to rank $m$} if the following conditions hold.
\begin{enumerate}
\item There exists a number $l \in \mathbb{Z}$ such that $\psi=\phi^l$ is proper reducible.
\item The mapping class $\psi$ induces a mapping class in the cluster modular group associated with the seed of mutable rank $m$ obtained by the cluster reduction with respect to a fixed point set $\{\alpha_1, \dots, \alpha_{n-m}\}$ of $\psi$.
\end{enumerate}

A cluster-reducible class $\phi$ of infinite order is called a \emph{cluster Dehn twist} if it is cluster-reducible to rank 2. Namely, there exists a number $l \in \mathbb{Z}$ and a subset $\{\alpha_1, \dots, \alpha_{n-2}\} \subset V(\C_{|\bi|})$ of vertices which is fixed by $\phi^l$ and contained in a cluster, where $n$ is the mutable rank of $\bi$.
\end{dfn}
A skew-symmetric seed is said to be \emph{connected} if the corresponding quiver is connected.
\begin{lem}\label{lem: cluster Dehn twists}
Let $\bi$ be a skew-symmetric connected seed of mutable rank $n \geq 3$. Suppose that a proper reducible class $\psi \in \Gamma_{|\bi|}$ has infinite order and there exists a subset $\{\alpha_1, \dots, \alpha_{n-2}\} \subset V(\C_{|\bi|})$ of vertices which is fixed by $\psi$ and contained in a cluster. Then the action of the cluster reduction $\psi' \in \Gamma_{|\bi'|}$ with respect to the invariant set $\{\alpha_1, \dots, \alpha_{n-2}\}$ on the $\A$-space is represented as follows:
\begin{equation}\label{eq: cluster Dehn twists}
(\psi')^*(A_0, A_1)=\left(A_1, \frac{C+A_1^2}{A_0} \right).
\end{equation}
Here $(A_0,A_1)$ denotes the remaining cluster coordinates of the $\A$-space under the cluster reduction, $C$ is a product of frozen variables.
\end{lem}

\begin{proof}
Take a cluster containing $\{\alpha_1, \dots, \alpha_{n-2}\}$. Let ${\bi}=(I, I_0, \epsilon)$ be the corresponding seed, and $i_j:=[\alpha_j] \in I$ the corresponding vertex for $j=1,\dots, n-2$. Then the cluster reduction produces a new seed ${\bi}':= (I, I_0 \sqcup \{i_1, \dots, i_{n-2}\}, \epsilon)$, whose mutable rank is 2. Label the vertices so that $I-I'_0=\{0,1\}$ and $I'_0=\{2, \dots, N-1\}$, where $I'_0:=I_0 \sqcup \{i_1, \dots, i_{n-2}\}$ and $N$ is the rank of the seed $\bi$. Note that $\psi=(0\ 1)\circ \mu_0 \in \Gamma_{|\bi|}$. Suitably relabeling if necessary, we can assume that $k:=\epsilon_{01} >0$. We claim that $k=2$. Since $\bi$ is connected, there exists a vertex $i \in I'_0$ such that $a:=\epsilon_{i0}\neq 0$ or $b:=\epsilon_{1i}\neq 0$. Since $\psi$ preserves the quiver, we compute that $a=b$ and $b-ak=-a$. Hence we conclude that $k=2$. Then from the definition of the cluster $\A$-transformation we have
\[
\psi^*(A_0, A_1)=\left(A_1, \frac{\prod_{i \in I'_0}A_i^{\epsilon_{i0}}+A_1^2}{A_0} \right)
\]
and $\psi^*(A_i)=A_i$ for all $i \in I'_0$, as desired.
\end{proof}

\begin{thm}\label{thm; cluster Dehn twists}
Let $\bi$ be a skew-symmetric connected seed of mutable rank $n \geq 3$ or the seed of type $L_2$. Then for each cluster Dehn twist $\phi \in \Gamma_{|\bi|}$, there exists a cluster-filling point $[G] \in P\A(\trop)$ such that we have
\[
\lim_{n \to \infty}\phi^{\pm n}(g)=[G]  \text{\ \ in $\overline{\A}$}
\]
for all $g \in \A(\pos)$.
\end{thm}

\begin{proof}
Assume that $n \geq 3$. There exists a number $\l$ such that $\psi:=\phi^l$ satisfies the assumption of \cref{lem: cluster Dehn twists}. Let as consider the following recurrence relation:
\[
\begin{cases}
a_0^{(n)}&=-a_1^{(n-1)}, \\
a_1^{(n)}&=-a_0^{(n-1)}+\log(C+e^{2a_1^{(n-1)}}),
\end{cases}
\]
where $C>0$ is a positive constant. It is the log-dynamics of \cref{eq: cluster Dehn twists}. Then one can directly compute that $a_0^{(n)}$, $a_1^{(n)}$ goes to infinity and $a_0^{(n)}\slash a_1^{(n)} \to 1$ as $n \to \infty$ for arbitrary initial real values. Hence we conclude that $\psi^n(g) \to [G]$ in $\overline{\A}$ for all $g \text{ in }\A(\pos)$, where $G \in \A(\trop)$ is the point whose coordinates are $a_0=a_1=1$, $a_i=0$ for all $i \in I'_0$. The proof for the negative direction is similar. The generator of $\Gamma_{L_2}$, which is cluster-pA, also satisfies the desired property.
\end{proof}

\begin{ex}[Dehn twists in the mapping class group]
Let $F=F_g^s$ be a hyperbolic surface with $s \geq2$. For an essential non-separating simple closed curve $C$, we denote the right hand Dehn twist along $C$ by $t_C \in MC(F)$. Consider an annular neighborhood $\mathcal{N}(C)$ of $C$, and slide two of punctures so that exactly one puncture lies on each boundary component of $\mathcal{N}(C)$. Let $\Delta$ be an ideal triangulation obtained by gluing the ideal triangulation of $\mathcal{N}(C)$ shown in \cref{fig; Dehn} and an arbitrary ideal triangulation of $F \setminus \mathcal{N}(C)$. This kind of a triangulation is given in ~\cite{Kash01}. Then the Dehn twist is represented as $t_C=(0\ 1)\circ \mu_0$, hence it is a cluster Dehn twist with $l=1$. Its action on the $\A$-space is represented as
\[
\phi_1^*(A_0, A_1, A_2, A_3)=\left(A_1,\frac{A_2A_3+A_1^2}{A_0}, A_2, A_3\right).
\]
\end{ex}
\begin{figure}
\[
\begin{tikzpicture}
\draw(0,0) circle(0.5cm) node[midway,right]{2};
\draw(0,0) circle(2cm);
\draw(0,0) circle(1cm)[thick];
\path (0,-1.2) node[circle]{$C$};
\draw (0.5,0) -- (2,0) node[midway,above] {0};
\draw (0.5,0) to[out=30, in=0] (0,1.5) to[out=180, in=90]  node[above] {1} (-1.5,0) to[out=270, in=180] (0,-1.5) to[out=0, in=210] (2.0,0) ;
\path (2,0) node[right]{3};
\end{tikzpicture}
\]
\caption{ideal triangulation of $\mathcal{N}(C)$}
\label{fig; Dehn}
\end{figure}
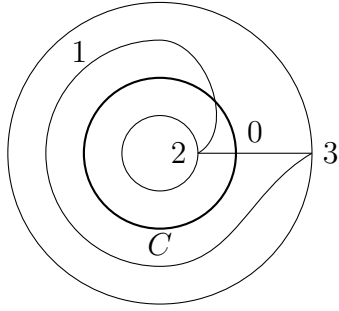

\begin{ex}[Type $X_7$]
Let us consider the seed of \emph{type $X_7$}. The mapping class $\phi_1:=(1\ 2)\circ \mu_1 \in \Gamma_{X_7}$ is a cluster Dehn twist, whose action on the $\A$-space is represented as
\[
\phi_1^*(A_0, A_1, A_2)=\left(A_0, A_2,\frac{A_0+A_2^2}{A_1}\right).
\]
\end{ex}

For a general cluster-pA class, we only know that it has at least one fixed point on the tropical boundary $P(\X(\trop)\backslash \X(\trop)_+$ from \cref{redX}. It would be interesting to find an analogue of the \emph{pA-pair} for a cluster-pA class which satisfies an appropriate condition, as we find in the surface theory (see \cref{classical NT}).

\section{Basic examples: seeds associated with triangulated surfaces}\label{section: Teich}
In this section we describe an important family of examples strongly related to the \Teich theory, following ~\cite{FST08}. A geometric description of the positive real parts and the tropical spaces associated with these seeds is presented in \cref{appendix: Teich}, which is used in \cref{Teich are Teich,comparison}.
In \cref{Teich are Teich} we prove that these seeds are of \Teich type. In \cref{comparison}, we compare the Nielsen-Thurston types defined in \cref{section: NT types} with the classification of mapping classes. In these cases, the characterization of periodic classes described in \cref{periodic} is complete. We show that cluster-reducible classes are reducible.

\subsection{Definition of the seed}\label{def: Teich seed} 
A \emph{marked hyperbolic surface} is a pair $(F, M)$, where $F=F_{g,b}^{p}$ is an oriented surface of genus $g$ with $p$ punctures and $b$ boundary components satisfying $6g-6+3b+3p+D>0$ and $p+b>0$, and $M \subset \partial F$ is a finite subset such that each boundary component has at least one point in $M$. The punctures together with elements of $M$ are called \emph{marked points}. We denote a marked hyperbolic surface by $F_{g, \vec{\delta}}^p$, where $\vec{\delta}=(\delta_1, \cdots, \delta_b)$, $\delta_i:= |M \cap \partial_i|$ indicates the number of marked points on the $i$-th boundary component. A connected component of $\partial F \backslash M$ is called a \emph{boundary segment}. We denote the set of boundary segments by $B(F)$, and fix a numbering on its elements. Note that $|B(F)|=D$, where $D:=\sum_{i=1}^b \delta_i$.
\begin{dfn}[the seed associated with an ideal triangulation]\ {} 

\begin{enumerate}
\item An ideal arc on $F$ is an isotopy class of an embedded arc connecting marked points, which is neither isotopic to a puncture, a marked point, nor an arc connecting two consecutive marked points on a common boundary component. An ideal triangulation of $F$ is a family $\Delta=\{ \alpha_{i} \}_{i=1}^{n}$ of ideal arcs, such that each connected component of $F \backslash \bigcup \alpha_{i}$ is a triangle whose vertices are marked points of $F$. One can verify that such a triangulation exists and that $n=6g-6+3r+3s+D$ by considering the Euler characteristic.
\item For an ideal triangulation $\Delta$ of $F$, we define a skew-symmetric seed $\bi_\Delta=(\Delta\cup B(F), B(F), \epsilon=\epsilon_\Delta)$ as follows. For an arc $\alpha$ of $\Delta$ which is contained in a self-folded triangle in $\Delta$ as in \cref{fig:self-folded}, let $\pi_\Delta(\alpha)$ be the loop enclosing the triangle. Otherwise we set $\pi_\Delta(\alpha):=\alpha$. Then for a non-self-folded triangle $\tau$ in $\Delta$, we define 
\[
\epsilon_{ij}^\tau:=
\begin{cases}
1, & \text{if $\tau$ contains $\pi_\Delta(\alpha_i)$ and $\pi_\Delta(\alpha_j)$ on its boundary in the clockwise order,} \\
-1, &  \text{if the same holds, with the anti-clockwise order,} \\
0, & \text{otherwise.}
\end{cases}
\]
Finally we define $\epsilon_{ij}:= \sum_\tau \epsilon_{ij}^\tau$, where the sum runs over non-self-folded triangles in $\Delta$.
\end{enumerate}
\end{dfn}

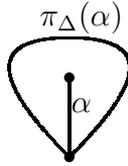
\begin{figure}[h] 
\begin{center} 
\setlength{\unitlength}{1.5pt} 
\begin{picture}(40,27)(-20,3) 
\thicklines 

\qbezier(0,0)(-30,30)(0,30)
\qbezier(0,0)(30,30)(0,30)
\put(0,0){\line(0,1){20}}
 
\put(3,13){\makebox(0,0){$\alpha$}}
\put(3,35){\makebox(0,0){$\pi_\Delta(\alpha)$}}
\multiput(0,0)(0,20){2}{\circle*{2}} 
\end{picture} 

\end{center} 
\caption{Self-folded triangle} 
\label{fig:self-folded} 
\end{figure} 

For an arc $\alpha$ of an ideal triangulation $\Delta$ which is a diagonal of an immersed quadrilateral in $F$ (in this case the quadrilateral is unique), we get another ideal triangulation $\Delta^{'}:= (\Delta \backslash \{ \alpha \})\cup \{ \beta \}$ by replacing $\alpha$ by the other diagonal $\beta$ of the quadrilateral. We call this operation the \emph{flip} along the arc $\alpha$.
One can directly check that the flip along the arc $\alpha_k$ corresponds to the mutation of the corresponding seed directed to the vertex $k$.

\begin{thm}[~\cite{Harer86,Hatcher91,Penner})]\label{Whitehead}
Any two ideal triangulations of $F$ are connected by a finite sequence of flips and relabellings.
\end{thm}
Hence the equivalence class of the seed $\bi_\Delta$ is determined by the marked hyperbolic surface $F$, independent of the choice of the ideal triangulation. We denote the resulting cluster ensemble by $p=p_{F}: \A_{F} \to \X_{F}$, the cluster modular group by $\Gamma_{F}:=\Gamma_{|\bi_\Delta|}$, \emph{etc}. The rank and the mutable rank of the seed $\bi_\Delta$ are $N=|\Delta \cup B(F)|=6g-6+3b+3p+2D$ and $n=|\Delta|=6g-6+3b+3p+D$, respectively.

Though a flip induces a mutation, not every mutation is realized by a flip. Indeed, the existence of an arc contained in a self-folded triangle prevents us from performing the flip along such an arc. Therefore we generalize the concept of ideal triangulations, following ~\cite{FST08}.

\begin{dfn}[tagged triangulations]\ {}
\begin{enumerate}
\item A tagged arc on $F$ is an ideal arc together with a label $\{\text{plain, notched}\}$ assigned to each of its end, satisfying the following conditions:
\begin{itemize}
\item the arc does not cut out a once-punctured monogon as in \cref{fig:self-folded},
\item each end which is incident to a marked point on the boundary is labeled plain, and
\item both ends of a loop are labeled in the same way.
\end{itemize}
The labels are called \emph{tags}. 
\item The \emph{tagged arc complex} $\Arc$ is the clique complex for an appropriate  compatibility relation on the set of tagged arcs on $F$. Namely, the vertices are tagged arcs and the collection $\{\alpha_1, \cdots, \alpha_k\}$ spans a $k$-simplex if and only if they are mutually compatible. See, for the definition of the compatibility, ~\cite{FST08}. The maximal simplices are called \emph{tagged triangulations} and the codimension 1 simplices are called \emph{tagged flips}.
\end{enumerate}
\end{dfn}
Note that if the surface $F$ has no punctures, then each tagged triangulation has only plain tags. 
If $F$ has at least two punctures or it has non-empty boundary, then the tagged arc complex typically contains a cycle (which we call a \emph{$\diamondsuit$-cycle}) shown in the right of \cref{fig; diamond-cycle}. Here by convention, the plain tags are omitted in the diagram while the notched tags are represented by the $\bowtie$ symbol. Compare with the ordinary arc complex, shown in the left of \cref{fig; diamond-cycle}. Compatibility relation implies that for a compatible set of tagged arcs and each puncture $a$, either one of the followings hold.
\begin{enumerate}
\item[(a)]All tags at the puncture $a$ are plain.
\item[(b)]All tags at the puncture $a$ are notched.
\item[(c)]The number of arcs incident to the puncture $a$ is at most two, and their tags at the puncture $a$ is different.
\end{enumerate}

\begin{figure}[h]
\unitlength 0.7mm
\begin{subfigure}[b]{.495\linewidth}
\centering
{\begin{picture}(100,100)(-20,-20)
\put(10,30){\line(4, 3){20}} \put(0,30){\makebox(0,0)[cc]{
\begin{tikzpicture}[scale=0.5]
\fill (0,0) circle(2pt) coordinate(A);
\fill (0,2) circle(2pt) coordinate(B);
\fill (0,4) circle(2pt) coordinate(C);
\coordinate (D) at (-1,2);
\coordinate (E) at (1,2);
\coordinate (F) at (0,2.5);
\draw (A) to[out=90, in=190] (F) to[out=0, in=90] (A) ;
\draw (A)--(B) ;
\draw (A) to[out=90, in=270] (D) to[out=90, in=270] (C);
\draw (A) to[out=90, in=270] (E) to[out=90, in=270] (C);
\path (0,-1) node[above]{$\Delta_1^\circ$};
\end{tikzpicture}
}}
 \put(30,65){\makebox(0,0)[cc]{
\begin{tikzpicture}[scale=0.5]
\fill (0,0) circle(2pt) coordinate(A);
\fill (0,2) circle(2pt) coordinate(B);
\fill (0,4) circle(2pt) coordinate(C);
\coordinate (D) at (-1,2);
\coordinate (E) at (1,2);
\draw (A)--(B);
\draw (B)--(C);
\draw (A) to[out=90, in=270] (D) to[out=90, in=270] (C);
\draw (A) to[out=90, in=270] (E) to[out=90, in=270] (C);
\path (0,5) node[below]{$\Delta_2^\circ=\Delta_4^\circ$};
\end{tikzpicture}
}}

\put(50,30){\line(-4,3){20}}\put(62,30){\makebox(0,0)[cc]{
\begin{tikzpicture}[scale=0.5]
\fill (0,0) circle(2pt) coordinate(A);
\fill (0,2) circle(2pt) coordinate(B);
\fill (0,4) circle(2pt) coordinate(C);
\coordinate (D) at (-1,2);
\coordinate (E) at (1,2);
\coordinate (F) at (0, 1.5);
\draw (B)--(C) ;
\draw (C) to[out=270, in=180] (F) to[out=0, in=270] (C);
\draw (A) to[out=90, in=270] (D) to[out=90, in=270] (C);
\draw (A) to[out=90, in=270] (E) to[out=90, in=270] (C);
\path (0,-1) node[above]{$\Delta_3^\circ$};
\end{tikzpicture}
}}

\end{picture}}
\end{subfigure}
\begin{subfigure}[b]{.495\linewidth}
\centering
{\begin{picture}(100,100)(-20,-20)
\put(10,30){\line(4, 3){20}} \put(0,30){\makebox(0,0)[cc]{
\begin{tikzpicture}[scale=0.5]
\fill (0,0) circle(2pt) coordinate(A);
\fill (0,2) circle(2pt) coordinate(B);
\fill (0,4) circle(2pt) coordinate(C);
\coordinate (D) at (-1,2);
\coordinate (E) at (1,2);
\draw (A) to[out=90, in=190] (B) ;
\draw (A)--(B) ;
\draw (A) to[out=90, in=270] (D) to[out=90, in=270] (C);
\draw (A) to[out=90, in=270] (E) to[out=90, in=270] (C);
\path (B) node[below]{$\bowtie$};
\path (0,-1) node[above]{$\Delta_1$};
\end{tikzpicture}
}}
\put(10,30){\line(4,-3){20}} \put(26,62){\makebox(0,0)[cc]{
\begin{tikzpicture}[scale=0.5]
\fill (0,0) circle(2pt) coordinate(A);
\fill (0,2) circle(2pt) coordinate(B);
\fill (0,4) circle(2pt) coordinate(C);
\coordinate (D) at (-1,2);
\coordinate (E) at (1,2);
\draw (A)--(B)node[midway,right]{$\alpha$};
\draw (B)--(C)node[midway,right]{$\beta$};
\draw (A) to[out=90, in=270] (D) to[out=90, in=270] (C);
\draw (A) to[out=90, in=270] (E) to[out=90, in=270] (C);
\path (-1,2) node[left]{$\Delta_4$};
\end{tikzpicture}
}}

\put(50,30){\line(-4,3){20}}\put(62,30){\makebox(0,0)[cc]{
\begin{tikzpicture}[scale=0.5]
\fill (0,0) circle(2pt) coordinate(A);
\fill (0,2) circle(2pt) coordinate(B);
\fill (0,4) circle(2pt) coordinate(C);
\coordinate (D) at (-1,2);
\coordinate (E) at (1,2);
\draw (B) to[out=10, in=270]  (C);
\draw (B)--(C) ;
\draw (A) to[out=90, in=270] (D) to[out=90, in=270] (C);
\draw (A) to[out=90, in=270] (E) to[out=90, in=270] (C);
\path (B) node[above]{$\bowtie$};
\path (0,-1) node[above]{$\Delta_3$};
\end{tikzpicture}
}}
\put(50,30){\line(-4,-3){20}}\put(26,-3){\makebox(0,0)[cc]{
\begin{tikzpicture}[scale=0.5]
\fill (0,0) circle(2pt) coordinate(A);
\fill (0,2) circle(2pt) coordinate(B);
\fill (0,4) circle(2pt) coordinate(C);
\coordinate (D) at (-1,2);
\coordinate (E) at (1,2);
\draw (A)--(B)  --(C);
\draw (A) to[out=90, in=270] (D) to[out=90, in=270] (C);
\draw (A) to[out=90, in=270] (E) to[out=90, in=270] (C);
\path (B) node[above]{$\bowtie$};
\path (B) node[below]{$\bowtie$};
\path (-1,2) node[left]{$\Delta_2$};
\end{tikzpicture}
}}
\end{picture}}
\end{subfigure}
\caption{ $\diamondsuit$-cycle}
\label{fig; diamond-cycle}

\end{figure}
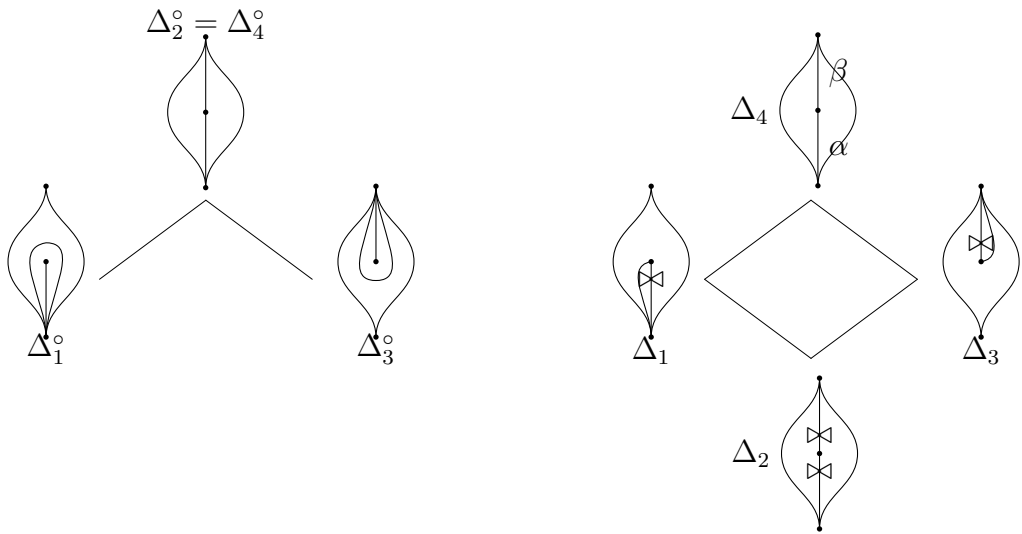


\begin{dfn}[the seed associated with a tagged triangulation]\ {}
For a tagged triangulation $\Delta$, let $\Delta^{\circ}$ be an ideal triangulation obtained as follows:
\begin{itemize}
\item replace all tags at a puncture $a$ of type (b) by plain ones, and
\item for each puncture $a$ of type (c), replace the arc $\alpha$ notched at $a$ (if any) by a loop enclosing $a$ and $\alpha$.
\end{itemize}
A tagged triangulation $\Delta$ whose tags are all plain is naturally identified with the corresponding ideal triangulation $\Delta^\circ$. For a tagged triangulation $\Delta$ with a fixed numbering on the member arcs, we define a skew-symmetric seed by ${\bi}_\Delta =(\Delta \cup B(F),B(F), \epsilon:= \epsilon_{\Delta_{\circ}})$.

\end{dfn}

Then we get a complete description of the cluster complex associated with the seed $\bi_\Delta$ in terms of tagged triangulations:
\begin{thm}[Fomin-Shapiro-Thurston ~\cite{FST08} Proposition 7.10, Theorem 7.11]\label{Arc}\ {}
For a marked hyperbolic surface $F=\surf$, the tagged arc complex has exactly two connected components (all plain/all notched) if $F=F_{g,0}^1$, and otherwise is  connected. The cluster complex associated with the seed $\bi_\Delta$ is naturally identified with a connected component of the tagged arc complex $\Arc$ of the surface $F$. Namely,
\[
\begin{cases}
\C_{F} \cong \mathrm{Arc}(F) & \text{if $F=F_{g,0}^1$},\\
\C_{F} \cong \Arc & \text{otherwise}.
\end{cases}
\]
\end{thm}
The coordinate groupoid of the seed $\bi_\Delta$ is denoted by $\mathcal{M}^{\bowtie}(F)$, and called \emph{tagged modular groupoid}. The subgroupoid $\mathcal{M}(F)$ whose objects are ideal triangulations and morphisms are (ordinary) flips is called the \emph{modular groupoid}, which is described in ~\cite{Penner}.
Next we see that the cluster modular group associated with the seed $\bi_\Delta$ is identified with some extension of the mapping class group.

\begin{dfn}[the tagged mapping class group]
The group $\{\pm 1\}^p$ acts on the tagged arc complex by alternating the tags at each puncture. The mapping class group naturally acts on the tagged arc complex, as well on the group $\{\pm 1\}^p$ by $(\phi_*\epsilon)(a):=\epsilon(\phi(a))$. 
Then the induced semidirect product $MC^{\bowtie}(F):=MC(F) \ltimes \{\pm 1\}^p$ is called the \emph{tagged mapping class group}. The tagged mapping class group naturally acts on the tagged arc complex.
\end{dfn}

\begin{prop}[Bridgeland-Smith ~\cite{BS15} Proposition 8.5 and 8.6]\label{MCG}
The cluster modular group associated with the seed $\bi_\Delta$ is naturally identified with the subgroup of the tagged mapping class group $MC^{\bowtie}(F)$ of $F$ which consists of the elements that preserve connected components of $\Arc$.
Namely,
\[
\begin{cases}
\Gamma_{F} \cong MC(F) & \text{if $F=F_{g,0}^1$},\\
\Gamma_{F} \cong MC^{\bowtie}(F) & \text{otherwise}.
\end{cases}
\]
\end{prop}

We give a sketch of the construction of the isomorphism here, for later use. 

\begin{proof}[Sketch of the construction]
Let us first consider the generic case $F\neq F_g^1$. Fixing a tagged triangulation $\Delta$, we can think of the cluster modular group as $\Gamma_{F}=\pi_1(\mathcal{M}^{\bowtie}(F), \Delta)$.
For a mapping class $\psi=(\phi, \epsilon) \in MC^{\bowtie}(F)$, there exists a sequence of tagged flips $\mu_{i_1}, \cdots, \mu_{i_k}$ from $\Delta$ to $\epsilon\cdot\phi^{-1}(\Delta)$ by \cref{Arc}. Since both $\phi$ and $\epsilon$ preserves the exchange matrix of the tagged triangulation, there exists a seed isomorphism $\sigma: \epsilon\cdot \phi^{-1}(\Delta) \to \Delta$. Then $I(\psi):= \sigma \circ \mu_{i_k} \cdots \mu_{i_1} $ defines an element of the cluster modular group. Hence we get a map $I: MC^{\bowtie}(F) \to \Gamma_{F}$, which in turn gives an isomorphism. Since each element of $MC(F)$ preserves the tags, the case of $F=F_g^1$ is clear.
\end{proof}

\subsection{The seed associated with an ideal triangulation is of \Teich type}\label{Teich are Teich}
Let $\Delta$ be an ideal triangulation of a marked hyperbolic surface $F$ and $\bi_\Delta$ the associated seed.

\begin{thm}\label{prop: Teich are Teich}
The seed $\bi_\Delta$ is of \Teich type.
\end{thm}

\begin{proof}

Condition (T1).
We claim that the action of the cluster modular group on each positive space is properly discontinuous. Then the assertion follows from \cref{growth}. First consider  the action on the $\X$-space. By \cref{action coincide}, the action of the subgroup $MC(F) \subset \Gamma_{F}$ on the $\X$-space $\X(\pos)$ coincide with the geometric action. Hence this action of $MC(F)$ is properly discontinuous, as is well-known. See, for instance, ~\cite{FM}. From the definition of the action of $\Gamma_{F}=MC^{\bowtie}(F)$ on the tagged arc complex and \cref{action}, one can verify that an element $(\phi,\epsilon) \in MC^{\bowtie}(F)$ acts on the positive $\X$-space as $(\phi,\epsilon)g=\phi(\iota(\epsilon)g)$, where $\iota(\epsilon):=\prod_{\epsilon(a)=-1} \iota_a$ is a composition of the involutions defined in \cref{def; involution}. Now suppose that there exists a compact set $K \subset \X(\pos)$ and an infinite sequence $\psi_m=(\phi_m,\epsilon_m) \in \Gamma_{F}$ such that $\psi_m(K) \cap K \neq \emptyset$. Since $\{\pm1\}^p$ is a finite group, $(\phi_m) \subset MC(F)$ is an infinite sequence and there exists an element $\epsilon \in \{\pm1\}^p$ such that $\epsilon_m=\epsilon$ for infinitely many $m$. Hence we have \[
\emptyset \neq \psi_m(K) \cap K =\phi_m(\iota(\epsilon)K) \cap K \subset \phi_m(\iota(\epsilon)K \cup K) \cap (\iota(\epsilon)K \cup K)
\]
for infinitely many $m$, which is a contradiction to the proper discontinuity of the action of $MC(F)$. Hence the action of $\Gamma_{F}$ on the $\X$-space is properly discontinuous. The action on the $\A$-space is similarly shown to be properly discontinuous. Here the action of $\epsilon$ is described as $\iota'(\epsilon):=\prod_{\epsilon(a)=-1} \iota'_a$, where $\iota'_a$ is the involution changing the horocycle assigned to the puncture $a$ to the conjugated one (see ~\cite{FoT}).

Condition (T2).
Note that for a tagged triangulation $\Delta=\{\gamma_1,\dots,\gamma_N\}$ without digons as in the left of \cref{Delta1} in \cref{appendix: Teich}, the map $\Psi_\Delta|[S_\Delta]$ is given by $\Psi([w_1,\dots,w_N])=(\bigsqcup w_j \gamma_j, \pm)$, where the sign at a puncture $p$ is defined to be $+1$ if the tags of arcs at $p$ are plain, and $-1$ if the tags are notched. Then on the image of these maps, the equivalence condition holds. Let us consider the tagged triangulation $\Delta_j$ in the $\diamondsuit$-cycle, see \cref{fig; diamond-cycle}. From the definition of the tropical $\X$-transformations, we have
\begin{align*}
\begin{cases}
x_{\Delta_2}(\alpha)=-x_{\Delta_1}(\alpha) \\
x_{\Delta_2}(\beta)=x_{\Delta_1}(\beta) 
\end{cases}
&
\begin{cases}
x_{\Delta_3}(\alpha)=-x_{\Delta_1}(\alpha) \\
x_{\Delta_3}(\beta)=-x_{\Delta_1}(\beta) 
\end{cases}
&
\begin{cases}
x_{\Delta_4}(\alpha)=x_{\Delta_1}(\alpha) \\
x_{\Delta_4}(\beta)=-x_{\Delta_1}(\beta).
\end{cases}
\end{align*}
Hence the equivalence condition on the image of the $\diamondsuit$-cycle holds.
\end{proof}

\subsection{Comparison with the Nielsen-Thurston classification of elements of the mapping class group}\label{comparison}
Let $F$ be a hyperbolic surface of type $F_g^1$ or $F_{g,\vec{\delta}}$ throughout this subsection.
Recall that in this case we have $\Gamma_{F} \cong MC(F)$ and $\C_{F} \cong \mathrm{Arc}(F)$, see \cref{MCG,Arc}. 
Let us recall the Nielsen-Thurston classification.
\begin{dfn}[Nielsen-Thurston classification]\label{classical NT}
A mapping class $\phi \in MC(F)$ is said to be
\begin{enumerate}
\item \emph{reducible} if it fixes an isotopy class of a finite union of mutually disjoint simple closed curves on $F$, and
\item \emph{pseudo-Anosov (pA)} if there is a pair of mutually transverse filling laminations $G_\pm \in \ML$ and a scalar factor $\lambda >0$ such that $\phi(G_\pm)=\lambda^{\pm 1}G_\pm$. The pair of projective laminations $[G_\pm]$ is called the \emph{pA-pair} of $\phi$.
\end{enumerate}
\end{dfn}
Here a lamination $G \in \wML$ is said to be \emph{filling} if each component of $F \backslash G$ is unpunctured or once-punctured polygon. 
It is known (see, for instance, ~\cite{FLP}) that each mapping class is at least one of  periodic, reducible, or pA, and a pA class is neither periodic nor reducible. Furthermore a mapping class $\phi$ is reducible if and only if it fixes a non-filling projective lamination, and is pA if and only if it satisfies $\phi(G)=\lambda G$ for some filling lamination $G\in \ML$ and a scalar $\lambda>0$, $\neq 1$. A pA class $\phi$ has the following asymptotic behavior of orbits in $\PML$: for any projective lamination $[G] \in \PML$ we have $\lim_{n \to\infty}\phi^{\pm n}[G]=[G_\pm]$.

We shall start with periodic classes. In this case the characterization of periodic classes described in \cref{periodic} is complete:

\begin{prop}\label{periodic mapping class}
For a mapping class $\phi \in \Gamma_{F}$, the following conditions are equivalent.
\begin{enumerate}
{\renewcommand{\labelenumi}{(\roman{enumi})}
\item The mapping class $\phi$ fixes a cell $C \in \C$ of finite type.
\item The mapping class $\phi$ is periodic.
\item The mapping class $\phi$ has fixed points in $\A_{F}(\pos)$ and $\X_{F}(\pos)$}.
\end{enumerate}
\end{prop}

\begin{lem}\label{lem: reduced cluster complex}
The cells of finite type (see \cref{periodic}) in the cluster complex are in one-to-one correspondence with ideal cell decompositions of $F$. Here an \emph{ideal cell decomposition} is a family $\Delta=\{\alpha_i\}$ of ideal arcs such that each connected component of $F\backslash \bigcup \alpha_i$ is a polygon.
\end{lem}

\begin{proof}
Let $C=(\alpha_1, \dots, \alpha_k)$ be a cell in the cluster complex, which is represented by a family of ideal arcs. Suppose that $\{\alpha_1, \dots, \alpha_k\}$ is an ideal cell decomposition. Then supercells of $C$ are obtained by adding some ideal arcs on the surface $F\backslash \bigcup_{i=1}^k \alpha_i$ to $\{\alpha_1, \dots, \alpha_k\}$, which are finite since such an ideal arc must be a diagonal of a polygon. Conversely suppose that $\{\alpha_1, \dots, \alpha_k\}$ is not an ideal cell decomposition. Then there exists a connected component $F_0$ of $F\backslash \bigcup_{i=1}^k \alpha_i$ which has a half-twist or a Dehn twist in its mapping class group. Hence $F_0$ has infinitely many ideal triangulations, consequently $C$ has infinitely many supercells. 
\end{proof}

\begin{proof}[Proof of \cref{periodic mapping class}]
It suffices to show that the condition $(\mathrm{iii})$ implies the condition $(\mathrm{i})$. Let $\C^*$ denote the union of all cells of finite type in the cluster complex. In view of \cref{lem: reduced cluster complex}, Penner's convex hull construction (~\cite{Penner}Chapter 4) gives a mapping class group equivariant isomorphism 
\[
\C^* \cong \wTF \slash \mathbb{R}_{>0},
\]
from which the assertion follows.
\end{proof}

Next we focus on cluster-reducible classes and their relation with \emph{reducible classes}. Observe that by \cref{prop: Teich are Teich} a mapping class is cluster-reducible if and only if it fixes an isotopy class of a finite union of mutually disjoint ideal arcs on $F$.

\begin{prop}\label{cluster-red mapping class}
The following holds.
\begin{enumerate}
\item A mapping class $\phi$ is cluster-reducible if and only if it fixes an unbounded lamination with \emph{real} weights $L=(\bigsqcup w_j \gamma_j, \pm)$, where $w_j \in \mathbb{R}$. If $\phi$ is proper reducible, then it induces a mapping class on the surface obtained by cutting $F$ along the multiarc $\bigsqcup \gamma_j$.
\item A cluster-reducible class is reducible. 
\item A filling lamination is cluster-filling. 
\end{enumerate}
\end{prop}

\begin{proof}\ {}

$(1)$.
The assertion follows from \cref{redX}$(2)$. Note that an element of $P\X(\trop)_+$ consists of elements of the form $L=(\bigsqcup w_j \gamma_j, \pm)$, where $w_j \in \mathbb{R}_{>0}$.

$(2)$.
Let $\phi \in MC(F)$ be a cluster-reducible class, $L=(\bigsqcup w_j \gamma_j, \pm)$ a fixed lamination, and $\bigsqcup \gamma_j$ the corresponding multiarc. One can pick representatives of $\phi$ and $\gamma$ so that $\phi(\gamma)=\gamma$ on $F$. Then by cutting $F$ along $\bigsqcup \gamma_j$, we obtain a surface $F'$ with boundary. Since $\phi$ fixes $\bigsqcup \gamma_j$, it induces a mapping class $\phi'$ on $F'$ which may permute the boundary components. Let $C'$ be the multicurve isotopic to the boundary of $F'$. Since $\phi'$ fixes $C'$, the preimage $C$ of $C'$ in $F$ is fixed by $\phi$. Therefore $\phi$ is reducible. 

$(3)$Let $G$ be a non-cluster-filling lamination. Let $\gamma$ be an ideal arc such that $a_\gamma(G)=0$. Then $G$ has no intersection with $\gamma$. Since $G$ has compact support, there is a twice-punctured disk which surrounds $\gamma$ and disjoint from $G$, which implies that $G$ is non-filling.
\end{proof}

\begin{ex}[a reducible class which is not cluster-reducible]\label{red and arc-irred}
Let $C$ be a non-separating simple closed curve in $F=F_g^p$, and $\phi \in MC(F)$ a mapping class given by the Dehn twist along $C$ on a tubular neighborhood $\mathcal{N}(C)$ of $C$ and a pA class on $F \backslash \mathcal{N}(C)$. Then $\phi$ is a reducible class which is not cluster-reducible.
\end{ex}

\begin{proof}
The reducibility is clear from the definition. Let $F' := F \backslash \mathcal{N}(C)$. If $\phi$ fixes an ideal arc contained in $F'$, then by \cref{cluster-red mapping class} we see that the restriction $\phi|_{F'} \in MC(F')$ is reducible, which is a contradiction. Moreover since $\phi$ is the Dehn twist along $C$ near the curve $C$, it cannot fix ideal arcs which traverse the curve $C$. Hence $\phi$ is cluster-irreducible.
\end{proof}

\begin{ex}[a cluster-filling lamination which is not filling]
Let $C$ be a simple closed curve in $F=F_g^p$, and $\{P_j\}$ be a pants decomposition  of $F$ which contains $C$ as a decomposing curve: $F=\bigcup_j P_j$. For a component $P_j$ which contains a puncture, let $G_j \in \mathcal{ML}_0^+(P_j)$ be a filling lamination such that $i(G_j, C)=0$. For a component $P_j$ which does not contain any punctures, choose an arbitrary lamination $G_j \in \mathcal{ML}_0^+(P_j)$. Then $G:= \bigsqcup_j G_j \sqcup C \in \ML$ is a cluster-filling lamination which is not filling. Indeed, each ideal arc $\alpha$ incident to a puncture. Let $P_j$ be the component which contains this puncture. Since $G_j \in \mathcal{ML}_0^+(P_j)$ is filling, it intersect with the arc $\alpha$: $i(\alpha,G_j)\neq 0$. In particular $i(\alpha, G)\neq 0$. Hence $G$ is cluster-filling. However, $G$ is not necessarily filling since the complement $F\setminus G$ in general contain a component $P_j$ without punctures, which is not a polygon.
\end{ex}
Just as the fact that a mapping class is reducible if and only if it fixes a non-filling projective lamination, we expect that a mapping class is cluster-reducible if and only if it fixes a non-cluster-filling projective lamination.

\appendix

\section{The seed of type $L_k$ is of \Teich type}\label{examples}
Here we show that the seed of type $L_k$ is of \Teich type. Recall that $\Gamma_{L_k} \cong \mathbb{Z}$ from \cref{example: periodic} and the generator $\phi$ is a cluster Dehn twist.

\begin{thm}\label{Lk theorem}
The seed of type $L_k$ is of \Teich type.
\end{thm}

\begin{proof}

Condition (T1).
Since the cluster complex $\C_{L_k}$ is homeomorphic to the real line and the cluster modular group acts by the shift, it suffices to show that each orbit of the generator $\phi$ is divergent. In the case of the $\X$-space, let us consider the following recurrence relation:
\begin{eqnarray*}
X_0^{(m)}&=&X_1^{(m-1)}(1+X_0^{(m-1)})^k \\
X_1^{(m)}&=&(X_0^{(m-1)})^{-1}.
\end{eqnarray*}
We claim that $\log X_0^{(m)}$ and $\log X_1^{(m)}$ diverges as $m \to \infty$. Setting $x_m:=\log X_0^{(m)}$ and deleting $X_1^{(m)}$, we have a 3-term recurrence relation
\[
x_m=-x_{m-2}+k\log(1+\exp x_{m-1}).
\]
Subtracting $x_{m-1}$ from the both sides, we have
\begin{equation}\label{recurrence y}
y_m=y_{m-1}+ f(x_{m-1})
\end{equation}
where we set $y_m:=x_m -x_{m-1}$ and $f(x):=k\log(1+\exp x)-2x$. Since $f(x)$ is positive, if $y_N$ is non-negative for some $N$, then $(y_n)_{n \geq N}$ is monotone increasing and
\[
x_n =x_{N-1}+\sum_{k=N}^{n}y_k \geq x_{N-1} +(n-N+1)y_1 \to +\infty
\]
as $n \to \infty$.
Therefore, it is enough to show that $y_M$ is non-negative for some $M$. 

Suppose that $y_m <0$ for all $m \geq 1$. Note that if $x_m \leq 0$ for some $m$, then $x_{m+2} >0> x_{m+1}$, hence $y_{m+2} >0$. Therefore it suffices to consider the case $x_m >0$ for all $m \geq 1$. Then $x_m$ is a decreasing sequence of positive numbers. In this case, from \cref{recurrence y} we have 
\[
y_m=y_1 + \sum_{k=1}^{m-1} f(x_k) \geq y_1 + (m-1)\min_{0 \leq x \leq x_1}f(x) \to \infty
\]
as $m \to \infty$, which is a contradiction.
Thus $\log X_0^{(m)}$ diverges to $+\infty$ and $\log X_1^{(m)}$ diverges to $-\infty$.
Hence the condition (T1) holds for the $\X$-space. We have proved the case of the $\A$-space in \cref{thm; cluster Dehn twists}.

Condition (T2).
In the tropical $\X$-coordinate $(x_0, x_1)$ associated with the seed ${\bi}_k$, the action of $\phi$ on the tropical $\X$-space is expressed as follows:
\[
\phi(x_0, x_1)=(x_1 + k\max\{0, x_0\}, -x_0).
\] 
To prove the condition (T2), we need to know the change of signs of tropical coordinates induced by the action described above. The following lemma follows from a direct calculation.

\begin{lem}\label{linear lemma}
Consider the following recurrence relation:
\[
\begin{cases}
x_0^{(m)}&=x_1^{(m-1)}+kx_0^{(m-1)}, \\
x_1^{(m)}&=-x_0^{(m-1)}.
\end{cases}
\]
Then we have that if $x_0^{(0)} >0$ and $x_0^{(0)}+x_1^{(0)} >0$, then $x_0^{(n)} >0$ and $x_0^{(m)}+x_1^{(m)} >0$ for all $m\geq0$.

In particular, the tropical action of $\phi$ on the cone $C_+:=\{ (x_0,x_1)| \text{ $x_0 >0$,  $x_0+x_1 >0$}\}$ is expressed by the linear transformation 
\begin{equation}\label{linear action positive}
\phi\begin{pmatrix}x_0 \\  x_1\end{pmatrix}=\begin{pmatrix}k &1 \\ -1 &0\end{pmatrix}\begin{pmatrix}x_0 \\  x_1\end{pmatrix}.
\end{equation}
\end{lem}

Let $L \in \X(\trop)_+ \backslash\{0\}$ be an arbitrary point, and $\bi$ a non-negative seed for $L$. Then $L=(x_0, x_1)$ and $x_0$, $x_1 \geq 0$ in the coordinate associated with $\bi$. In particular we have $x_0+x_1 >0$. 

If $x_0 >0$, by \cref{linear lemma} we have $x_1^{(m)}=-x_0^{(m-1)} <0$ for all $m \geq 1$, which implies that no seed other than $\bi$ is non-negative for $L$. 

If $x_0=0$, we have $x_1 >0$ and $x_0^{(1)} >0$, which implies that $\mu_0({\bi})$ is again a non-negative seed for $L$, while any other seeds are not non-negative for $L$ from the argument in the previous paragraph. Hence the condition (T2) holds.
\end{proof}

Next we study the asymptotic behavior of orbits of the generator $\phi$ of $\Gamma_{L_k}$ on the tropical $\X$-space, which may be related with that of general cluster Dehn twists.
\begin{prop}\label{Lk hyperbolic}
For $k \geq 2$, the generator $\phi$ of the cluster modular group $\Gamma_{L_k}$ has  unique attracting/repelling fixed points $[L_{\pm}] \in P\X(\trop)$ such that for all $L \in \X(\trop)$ we have 
\[
\lim_{m \to \infty}\phi^{\pm m}([L])=[L_{\pm}] \text{ in } P\X(\trop).
\]
\end{prop}

\begin{proof}
Note that $\phi^{-1}(x_0,x_1)=(-x_1, x_0+k\min\{0,x_1\})$.
By a similar argument as \cref{linear lemma}, we have that the cone $C_-:=\{(x_0,x_1)|\text{ $x_1<0$, $x_0+x_1<0$}\}$ is stable under $\phi^{-1}$ and on this cone 
\begin{equation}\label{linear action negative}
\phi^{-1}\begin{pmatrix}x_0 \\  x_1\end{pmatrix}=\begin{pmatrix}0 &-1 \\ 1 &k\end{pmatrix}\begin{pmatrix}x_0 \\  x_1\end{pmatrix}.
\end{equation}
Together with the fact that $\phi(-1,0)=(0,1)$ and $\phi(0,1)=(1,0)$, we see that for all $L \in \X(\trop)$, $\phi^{\pm N}(L) \in C_\pm$ for a sufficiently large number $N$. Then each orbit $(\phi^{n}(L))_{n \geq 0}$ projectively converges to the unique attracting fixed point $[L_+]$ of the linear action (\ref{linear action positive}), which is represented by $(k+\sqrt{k^2-4}, -2)$. Similarly $(\phi^{n}(L))_{n \leq 0}$ projectively converges to the unique repelling fixed point $[L_-]$ of the linear action (\ref{linear action negative}), which is represented by $(k-\sqrt{k^2-4}), -2)$. 
\end{proof}

\section{The positive real parts and the tropical spaces associated with the seed $\bi_\Delta$}\label{appendix: Teich}
Here we give a geometric description of the positive real parts and the tropical spaces associated with the seed $\bi_\Delta$ coming from an ideal triangulation $\Delta$ of a marked hyperbolic surface $F=\surf$.
Most of the contents of this section seems to be well-known to specialists, but they are scattered in literature. Therefore we tried to gather the results and give a coherent presentation of the data associated with $\bi_\Delta$.

\subsection{Positive spaces and the \Teich spaces}
Here we describe the positive real parts of $\A_{F}$ and $\X_F$ geometrically. Main references are ~\cite{FG07, FoT, Penner}.
For simplicity, we only deal with the case of empty boundary, $b=0$. The case of non-empty boundary is reduced to the case of empty boundary by duplicating the surface and considering the invariant subspace of the \Teich/lamination spaces under the natural involution. See, for details, ~\cite{Penner} section 2.

Let $F=F_g^p$ be a hyperbolic punctured surface. A non-trivial element $\gamma \in \pi_1(F)$ is said to be \emph{peripheral} if it goes around a puncture, and \emph{essential} otherwise. Let $\TF$ denote the Teichm\"uller space of all complete finite-area hyperbolic structures on $F$. Namely,
\[
T(F):=\rm{Hom}'(\pi_{1}(F), PSL_{2}(\mathbb{R}))/PSL_{2}(\mathbb{R}), 
\]
where $\rm{Hom}'(\pi_{1}(F), \PSL)$ consists of faithful representations $\rho:\pi_{1}(F)\to\PSL$ such that
\begin{enumerate}
\item the image of $\rho$ is a discrete subgroup of $PSL_2(\mathbb{R})$, and
\item it maps each peripheral loop to a parabolic element, essential one to hyperbolic ones.
\end{enumerate}
Note that each element $\rho \in T(F)$ determines a hyperbolic structure by $F \cong \mathbb{H}/\rho(\pi_1(F))$, where $\mathbb{H}:=\{z \in \mathbb{C} \mid \Im z>0\}$ is the upper half-plane.

\begin{dfn}[decorated \Teich space]
The trivial bundle $\wTF:= \TF \times \mathbb{R}_{>0}^{s}$ is called the \emph{decorated \Teich space}. Let $\varpi: \wTF \to \TF$ be the natural projection.
\end{dfn}
Here the fiber parameter determines a tuple of \emph{horocycles} centered at each punctures. Specifically, let $\mathbb{D}:=\{ w \in \mathbb{C} \mid |w|<1 \}$ be the Poincar\'e disc model of the hyperbolic plane. A \emph{horocycle} is a euclidean circle in $\mathbb{D}$ tangent to the boundary $\partial\mathbb{D}$. The tangent point is called the \emph{center} of the horocycle.
For a point $\tilde{g}=(g, (u_a)_{a=1}^p) \in \wTF$ and a puncture $a$, let $\tilde{a} \in \partial\mathbb{D}$ be a lift of $a$ with respect to the hyperbolic structure $g$ and $\tilde{h}_a(u_a)$ the horocycle centred at $\tilde{a}$ whose euclidean radius is given by $1/(1+u_a)$. Then $h_a(u_a):= \varpi(\tilde{h}_a(u_a))$ is a closed curve in $F$, which is independent of the choice of a lift $\tilde{a}$. We call it a horocycle in $F$.

Given a point $\tilde{\rho}=(\rho, (u_a)_{a=1}^p)$ of $\wTF$, we can associate a positive real number with each ideal arc $e$ as follows. Straighten $e$ to a geodesic in $F$ for the hyperbolic structure given by $\rho$. Take a lift $\tilde{e}$ to the universal cover $\mathbb{D}$. Then there is a pair of horocycles given by the fiber parameters $u_a$, centred at each of the endpoints of $\tilde{e}$. Let $\delta$ denote the signed hyperbolic distance of the segment of $\tilde{e}$ between these two horocycles, taken with a positive sign if and only if the horocycles are disjoint. Finally, define the $\A$-coordinate (which is called \emph{$\lambda$-length coordinate} in \cite{Penner}) of $e$ for $\tilde{\rho}$ to be $A_{e}(\tilde{\rho}):= \sqrt{e^{\delta/2}}$.
Then $A_e$ defines a function on $\wTF$. For an ideal triangulation $\Delta$ of $F$, we call the set $A_\Delta=(A_e)_{e\in \Delta}$ of functions the \emph{Penner coordinate} associated with $\Delta$.
\begin{prop}[Penner ~\cite{Penner} Chapter 2, Theorem 2.5]\label{decorated teich}

For any ideal triangulation $\Delta$ of $F$, the Penner coordinate
\[
A_{\Delta}: \wTF \to \mathbb{R}^{\Delta}_{>0}
\]
gives a real analytic diffeomorphism. Furthermore the Penner coordinates give rise to a positive space $\psi_{\A}: \mathcal{M}(F) \to \Pos(\mathbb{R})$. More precisely, the coordinate transformation with respect to the flip along an ideal arc $e\in\Delta$ is given by the positive rational maps shown in \cref{fig: A-coord}.
\end{prop}
\begin{figure}[h]
\unitlength 1mm
{\begin{picture}(100,50)(4,95)
\put(70,120){\line(1, 1){20}} \put(78,108){\makebox(0,0)[cc]{$A_d$}}
\put(70,120){\line(1,-1){20}} \put(78,132){\makebox(0,0)[cc]{$A_a$}}
\put(70,120){\line(1,0){40}} \put(90,120) {\makebox(0,0)[bc]{$ \frac{A_a A_c+A_b A_d}{A_e}$}}
\put(110,120){\line(-1,1){20}}\put(103,108){\makebox(0,0)[cc]{$ A_c$}}
\put(110,120){\line(-1,-1){20}}\put(103,132){\makebox(0,0)[cc]{$ A_b$}}
\put(90,95){\makebox(0,0)[cc]{$\Delta'$}}
\put(10,120){\line(1,1){20}} \put(18,108){\makebox(0,0)[cc]{$A_d$}}
\put(10,120){\line(1,-1){20}} \put(18,132){\makebox(0,0)[cc]{$ A_a$}}
\put(50,120){\line(-1,1){20}}\put(42,108){\makebox(0,0)[cc]{$ A_c$}}
\put(50,120){\line(-1,-1){20}}\put(42,132){\makebox(0,0)[cc]{$ A_b$}}
\put(30,140){\line(0,-1){40}}\put(32,120) {\makebox(0,0)[cc]{$ A_e$}}
\put(60,120){\makebox(0,0)[cc]{$\xrightarrow{\mu_e}$}}
\put(30,95){\makebox(0,0)[cc]{$\Delta$}}
\end{picture}}
\caption{ }
\label{fig: A-coord}
\end{figure}
In \cite{FoT}, the authors generalized the definition of the $\A$-coordinates to tagged arcs:

\begin{thm}[Fomin-Thurston \cite{FoT} Theorem 8.6]\label{A-extension}
The above functor extends to a positive space $\psi_{\A}^{\bowtie}: \mathcal{M}^{\bowtie}(F) \to \Pos(\mathbb{R})$ so that the positive real part is naturally identified with the decorated \Teich space $\wTF$, \emph{i.e.}, $\A(\mathbb{R}_{>0}) \cong \wTF$.
\end{thm}
The $\A$-coordinate for a tagged arc is obtained by modifying the $\A$-coordinate for the underlying ideal arc using \emph{conjugate horocycles}, see Section 7 of ~\cite{FoT}, for details. Here two horocycles $h$ and $\bar{h}$ on $F$ are called  \emph{conjugate} if the product of their length is $1$. Changing the tags at a puncture $a$ amounts to changing the horocycle centred at $a$ by the conjugate one. More precisely, let $\epsilon_a \in MC^{\bowtie}(F)$ be the element changing the tags at a puncture $a$, see \cref{MCG}. It acts on $\TF$ by changing the horocycle centred at $a$ by the conjugate one. 

\begin{dfn}[the enhanced \Teich space]
Let $T(F)'$ denote the \Teich space of all complete (not necessarily finite-area) hyperbolic structures on $F$. Namely,
\[
T(F)':=\rm{Hom}''(\pi_{1}(F), PSL_{2}(\mathbb{R}))/PSL_{2}(\mathbb{R}), 
\]
where $\rm{Hom}''(\pi_{1}(F), \PSL)$ consists of faithful representations $\rho:\pi_{1}(F)\to\PSL$ such that
\begin{enumerate}
\item the image of $\rho$ is a discrete subgroup of $PSL_2(\mathbb{R})$, and
\item it maps each peripheral loop to a parabolic \emph{or hyperbolic} element, essential one to a hyperbolic one.
\end{enumerate}
The \emph{enhanced \Teich space} $\hTF$ is defined to be a $2^p$-fold branched cover over $T(F)'$, whose fiber over a point $\rho \in T(F)'$ consists of data of an orientation on each puncture such that the corresponding peripheral loop is mapped to a hyperbolic element by $\rho$. 
\end{dfn}

Note that a point $\rho \in T(F)$ maps each peripheral loop to a parabolic element. Hence there is a natural embedding $\iota: T(F) \to \hTF$ (no orientations are needed).
For each ideal triangulation $\Delta$ of $F$, we define a coordinate on $\hTF$ as follows. Take an element $\rho \in T(F) \subset \hTF$, for simplicity. Each $e \in \Delta$ is the diagonal of a unique quadrilateral in $\Delta$. A lift of this quadrilateral is an ideal quadrilateral in $\mathbb{D}$, whose vertices are denoted by $x$, $y$, $z$ and $w$ in the clockwise order. Let $X_\Delta(e;\rho):= (x-w)(y-z)/(z-w)(x-y)$ be the cross ratio of these four points. The function $X_\Delta(e;-)$ can be extended to the enhanced \Teich space $\hTF$. We call the set $X_\Delta=(X_\Delta(e;-))$ of functions the \emph{Fock-Goncharov coordinate} associated with $\Delta$.

\begin{prop}[Fock-Goncharov ~\cite{FG07} Section 4.1]\label{prop: X-coord}

For any ideal triangulation $\Delta$ of $F$, the Fock-Goncharov coordinate
\[
X_{\Delta}: \hTF \to \mathbb{R}^{\Delta}_{>0}
\]
gives a real analytic diffeomorphism.
Furthermore the Fock-Goncharov coordinates give rise to a positive space $\psi_{\X}: \mathcal{M}(F) \to \Pos(\mathbb{R})$. More precisely, the coordinate transformation with respect to the flip along an ideal arc $e\in\Delta$ is given by the positive rational maps shown in \cref{fig: X-coord}.
\end{prop}
\begin{figure}[h]

\unitlength 1mm
{\begin{picture}(100,50)(4,95)
\put(70,120){\line(1, 1){20}} \put(69,108){\makebox(0,0)[cc]{$X_d(1+X_e^{-1})^{-1}$}}
\put(70,120){\line(1,-1){20}} \put(69,132){\makebox(0,0)[cc]{$X_a(1+X_e)$}}
\put(70,120){\line(1,0){40}} \put(90,120) {\makebox(0,0)[bc]{$ X_e^{-1}$}}
\put(110,120){\line(-1,1){20}}\put(110,108){\makebox(0,0)[cc]{$ X_c(1+X_e)$}}
\put(110,120){\line(-1,-1){20}}\put(112,132){\makebox(0,0)[cc]{$ X_b(1+X_e^{-1})^{-1}$}}
\put(90,95){\makebox(0,0)[cc]{$\Delta'$}}
\put(10,120){\line(1,1){20}} \put(18,108){\makebox(0,0)[cc]{$X_d$}}
\put(10,120){\line(1,-1){20}} \put(18,132){\makebox(0,0)[cc]{$ X_a$}}
\put(50,120){\line(-1,1){20}}\put(42,108){\makebox(0,0)[cc]{$ X_c$}}
\put(50,120){\line(-1,-1){20}}\put(42,132){\makebox(0,0)[cc]{$ X_b$}}
\put(30,140){\line(0,-1){40}}\put(32,120) {\makebox(0,0)[cc]{$ X_e$}}
\put(60,120){\makebox(0,0)[cc]{$\xrightarrow{\mu_e}$}}
\put(30,95){\makebox(0,0)[cc]{$\Delta$}}
\end{picture}}
\caption{ }
\label{fig: X-coord}
\end{figure}

\begin{prop}\label{X-extension}
The above functor extends to a positive space $\psi_{\X}^{\bowtie}: \mathcal{M}^{\bowtie}(F) \to \Pos(\mathbb{R})$ so that the positive real part is naturally identified with the enhanced \Teich space $\hTF$, \emph{i.e.}, $\X(\pos) \cong \hTF$.
\end{prop}

Although the above proposition seems to be well-known to specialists, we could not find any proof in literature. Therefore we give a proof here for completeness.

\begin{dfn}[the coordinates associated with a tagged triangulation]\label{def; involution}
We already have the Fock-Goncharov coordinates $X_\Delta : \X(\pos) \to \mathbb{R}^{\Delta}_{>0}$ with respect to any ideal triangulation $\Delta$. We define a coordinate system for any tagged triangulations by the following conditions:
\begin{enumerate}
\item Suppose two tagged triangulations $\Delta_1$, $\Delta_2$ coincide except for the tags at a puncture $a$. Then we set
\[
X_{\Delta_1}(g, \alpha_1)=X_{\Delta_2}(\iota_a(g),\alpha_2)
\]
for all $g \in \X(\pos)$, where $\alpha_i \in \Delta_i$ $(i=1,2)$ are the corresponding arcs, $\iota_a$ is the involution on $\X(\pos)$ reversing the fiber parameter of the cover assigned to the puncture $a$.
\item If the tags of a tagged triangulation $\Delta$ are all plain, then we set $X_\Delta(g, \alpha):=X_{\Delta^{\circ}}(g, \alpha^{\circ})$ for all $g \in \X(\pos)$ and $\alpha \in \Delta$. The right-hand side is the Fock-Goncharov coordinate for the ideal triangulation $\Delta^{\circ}$.
\item If a tagged triangulation $\Delta$ have a punctured digon shown in the left of  \cref{Delta1}, then we set $X_\Delta(g, \gamma):=X_{\Delta^{\circ}}(g, \gamma^{\circ})$ for $\gamma \neq \alpha$, and respecting the rule $(1)$ we set
\[
X_\Delta(g,\alpha)=X_{\Delta'}(\iota_a(g),\alpha')=X_{\Delta^{\circ}}(\iota_a(g),\beta^{\circ}).
\]
\end{enumerate}
\end{dfn}
\begin{figure}[h]
\unitlength 0.5mm
\[
\begin{tikzpicture}
\fill (0,0) circle(2pt) coordinate(A);
\fill (0,2) circle(2pt) coordinate(B);
\fill (0,4) circle(2pt) coordinate(C);
\coordinate (D) at (-1,2);
\coordinate (E) at (1,2);
\draw (A) to[out=90, in=190] node[midway,left]{$\alpha$}(B) ;
\draw (A)--(B) node[midway,right]{$\beta$};
\draw (A) to[out=90, in=270] (D) to[out=90, in=270] (C);
\draw (A) to[out=90, in=270] (E) to[out=90, in=270] (C);
\path (B) node[below]{$\bowtie$};
\path (-1,2) node[left]{$\Delta$};
\fill (5,0) circle(2pt) coordinate(A');
\fill (5,2) circle(2pt) coordinate(B');
\fill (5,4) circle(2pt) coordinate(C');
\coordinate (D') at (4,2);
\coordinate (E') at (6,2);
\draw (A') to[out=90, in=350] node[midway,right]{$\beta$}(B') ;
\draw (A')--(B') node[midway,left]{$\alpha$};
\draw (A') to[out=90, in=270] (D') to[out=90, in=270] (C');
\draw (A') to[out=90, in=270] (E') to[out=90, in=270] (C');
\path (B') node[below]{$\bowtie$};
\path (4,2) node[left]{$\Delta'$};
\fill (10,0) circle(2pt) coordinate(A'');
\fill (10,2) circle(2pt) coordinate(B'');
\fill (10,4) circle(2pt) coordinate(C'');
\coordinate (D'') at (9,2);
\coordinate (E'') at (11,2);
\draw (A'')--(B'') node[midway,left]{$\alpha^{\circ}$};
\draw (A'') to[out=90, in=270] (D'') to[out=90, in=270] (C'');
\draw (A'') to[out=90, in=270] (E'') to[out=90, in=270] (C'');
\draw (A'') to[out=90, in=180] (10,2.5) to[out=0, in=90] node[midway,right]{$\beta^\circ$} (A'') ;
\path (9,2) node[left]{$\Delta^{\circ}$};
\end{tikzpicture}
\]
\caption{}
\label{Delta1}
\end{figure}

The following is the key lemma to ensure that the above definition is well-defined, which is essentially a special case of Lemma 12.3 in ~\cite{FG06}:

\begin{lem}[Fock-Goncharov ~\cite{FG06}]\label{lem; involution}
In the notation of \cref{fig; diamond-cycle}, we have
\[
X_{\Delta_4}(g,\alpha)X_{\Delta_4}(\iota_a(g),\beta)=1
\]
for all $g \in \hTF$.
\end{lem}

\begin{proof}[Proof of \cref{X-extension}]
We need to show that each coordinate transformation $X_{\Delta_i}\circ X_{\Delta_j}^{-1}$ in the $\diamondsuit$-cycle coincides with the cluster $\X$-transformation with respect to the exchange matrices associated with the tagged triangulations.

$\Delta_1 \overset{\mu_\beta}{\longleftrightarrow} \Delta_4$.
Note that the coordinate transformations of $X_\gamma$ ($\gamma \neq \alpha$) coincide with the cluster transformations by \cref{prop: X-coord}. In particular we have $X_{\Delta_1}(g,\beta)=X_{\Delta_4}(g,\beta)^{-1}$. Hence by \cref{lem; involution} we have
\[
X_{\Delta_1}(g,\alpha)=X_{\Delta_1^\circ}(\iota_a(g),\beta^\circ)=X_{\Delta_4^\circ}(\iota_a(g),\beta^\circ)^{-1}=X_{\Delta_4^\circ}(g,\alpha^\circ)=X_{\Delta_4}(g,\alpha).
\]

$\Delta_1 \overset{\mu_\alpha}{\longleftrightarrow} \Delta_2$.
Note that the coordinate transformations of $X_\gamma$ ($\gamma \neq \beta$) coincide with the corresponding cluster transformations, and we have
\[
X_{\Delta_2}(g,\beta)=X_{\Delta_4}(\iota_a(g),\alpha)=X_{\Delta_1}(\iota_a(g),\alpha)=X_{\Delta_1^\circ}(g,\beta^\circ)=X_{\Delta_1}(g,\beta).
\]
The remaining cases follow from a symmetric argument.
\end{proof}

The monomial morphism between the positive spaces of the seed $\bi_\Delta$ coincides with $p= \iota \circ \varpi: \wTF \to \hTF$ (see, for instance, ~\cite{Penner} Chapter 1, Corollary 4.16(c)). In particular, the $\U$-space is naturally identified with the \Teich space,\emph{i.e.}, $\U(\pos)\cong T(F)$.

The mapping class group naturally acts on the \Teich space $T(F)$ and $T(F)'$ via the Dehn-Nielsen embedding ~\cite{FM} $MC(F) \to \mathrm{Out}(\pi_1(F))$. These two actions extend to $\wTF$ and $\hTF$ by permuting the fiber parameters according to the action on the punctures.

\begin{prop}[Penner ~\cite{Penner} Chapter 2, Theorem 2.10]\label{action coincide}
For any ideal triangulation $\Delta$ of $F$ and a mapping class $\phi \in \MC$, the following diagrams commute: 
\begin{align*}
\xymatrix{
\wTF \ar[d]_{\phi} \ar[r]^{A_{\Delta}}  & \mathbb{R}^{\Delta}_{>0}  \ar[d]^{I(\phi)} \\
\wTF \ar[r]_{A_{\Delta}} & \mathbb{R}^{\Delta}_{>0}, \\
}
&
\xymatrix{
\hTF \ar[d]_{\phi} \ar[r]^{X_{\Delta}}  & \mathbb{R}^{\Delta}_{>0}  \ar[d]^{I(\phi)} \\
\hTF \ar[r]_{X_{\Delta}} & \mathbb{R}^{\Delta}_{>0}. \\
}
\end{align*}
In particular,  these natural actions coincide with the action as a subgroup of the cluster modular group associated with the seed $\bi_\Delta$. Compare with \cref{action}.
\end{prop}

\subsection{Tropical spaces and the lamination spaces}

Next we describe the tropical spaces geometrically, following ~\cite{FG07}. As before, we focus on the case of empty boundary $b=0$.

\begin{dfn}
A \emph{decorated rational (bounded) lamination} on $F$ is an isotopy class of a disjoint union of simple closed curves in $F$ with rational numbers (called weights) assigned to each curve so that the weight is positive unless the corresponding curve is peripheral. Each curve is called a \emph{leaf} of the lamination.
\end{dfn}
We denote a decorated rational lamination by $L=\bigsqcup w_j\gamma_{j}$, and denote the set of decorated rational laminations by $\wLQ$.
Let $\LQ$ denote the set of decorated rational laminations with no peripheral leaves. 
There is a canonical projection $\varpi : \wLQ \to \LQ$ forgetting the peripheral leaves. 
Following ~\cite{FG07}, we associate a rational number with an ideal arc $e$.
For a decorated rational lamination $L=\bigsqcup w_j\gamma_{j}$, isotope each curve $\gamma_{j}$ so that the intersection with $e$ is minimal. Then define $a_e(L):=\sum_{j} w_{j} \#(\gamma_{j}\cap e)$.

\begin{prop}[Fock-Goncharov ~\cite{FG07} Section 3.2]
For any ideal triangulation $\Delta$ of $F$, the map
\begin{equation}
a_{\Delta} : \wLQ \to \mathbb{Q}^{\Delta} ; L \mapsto \{ a_e(L) \}_{e \in \Delta} \nonumber
\end{equation}
gives a bijection.
\end{prop}
For a flip along $e \in \Delta$, the corresponding change of the above coordinates coincide with the tropical cluster $\A$-transformation. Thus we call $a_\Delta$ the tropical $\A$-coordinate associated with $\Delta$.
Since the tropical cluster $\A$-transformation is continuous with respect to the standard topology on $\mathbb{Q}^N$we can define the \emph{real decorated lamination space} $\wLR$ as the completion of $\wLQ$ with respect to the topology induced by the tropical $\A$-coordinates. Similarly define $\LR$ as the completion of $\LQ$. Then we have a homeomorphism $a_\Delta: \wLR \to \mathbb{R}^\Delta$ for each ideal triangulation $\Delta$.
For each tagged arc we can extend the definition of the tropical $\A$-coordinate using the \emph{conjugate peripheral curves}, in analogy with the conjugate horocycles. Here two weighted peripheral curves on $F$ are called \emph{conjugate} if the sum of weights is $0$. Again, changing the tags amounts to changing the weighted peripheral curves by the conjugate one.
\begin{prop}[Fomin-Thurston \cite{FoT}]
The tropical space of the positive space $\psi_{\A}^{\bowtie}: \mathcal{M}^{\bowtie}(F) \to \Pos(\mathbb{R})$ given in \cref{A-extension} is naturally identified with the real decorated lamination space $\wLR$, \emph{i.e.}, $\A(\trop) \cong \wLR$.
\end{prop}

Although the geometric meaning of irrational points in $\wLR$ is not so clear from the above definition, we have the following result.

\begin{thm}[for instance, ~\cite{PH}]
There are natural PL homeomorphisms 
\begin{align*}
\LR \cong \ML&\ \text{and} \  \wLR \cong \wML, 
\end{align*}
where $\ML:=\mathcal{ML}_0(F) \cup \{\emptyset\}$ is the space of measured geodesic laminations with compact supports attached with the empty lamination, and $\wML:= \ML \times \mathbb{R}^s$ is a trivial bundle. Moreover, the bundle projection $\wML \to \ML$ coincides with the projection $\varpi: \wLR \to \LR$.
\end{thm}

\begin{dfn}
A \emph{rational unbounded lamination} consists of the following data:
\begin{enumerate}
\item an isotopy class of a disjoint union of simple closed curves and ideal arcs $\{\gamma_j\}_j$ in $F$ with positive rational weights $\{w_j\}$ assigned to each curve.
\item a tuple of orientations on each puncture to which some curves incident.
\end{enumerate}
We denote these data by $L=(\bigsqcup w_j \gamma_j, \pm)$.
\end{dfn}
Denote the set of rational unbounded laminations by $\hLQ$. We have a natural embedding $\iota: \LQ \hookrightarrow \hLQ$, where the orientation data is unnecessary since the leaves of a bounded lamination are not incident to any punctures.

\begin{prop}[Fock-Goncharov ~\cite{FG07} Section 3.1]
For any ideal triangulation $\Delta$ of $F$, there exists a natural bijection
\begin{equation}
x_{\Delta} : \hLQ \to \mathbb{Q}^{\Delta} \nonumber.
\end{equation}

\end{prop}

For a flip along $e \in \Delta$, the corresponding change of the above coordinates coincide with the tropical cluster $\X$-transformation.
By the continuity of the tropical cluster $\X$-transformations, we can define the real unbounded lamination space $\hLR$ as the completion of $\hLQ$. 

\begin{prop}[Fomin-Thurston \cite{FoT} Theorem 13.6]
The coordinate functor defined above naturally extends to the tagged modular groupoid $\mathcal{M}^{\bowtie}(F)$, and the tropical space of the positive space $\psi_{\X}: \mathcal{M}(F) \to \Pos(\mathbb{R})$ given in \cref{decorated teich} is naturally identified with the real decorated lamination space $\hLR$, \emph{i.e.}, $\X(\trop) \cong \hLR$.
\end{prop}
The extension is in the same manner as the one described in \cref{X-extension}.

\bibliographystyle{amsalpha}
\bibliography{reference}

\end{document}